\documentclass[a4paper,11pt]{article}
\usepackage{graphicx,epsfig}
\usepackage{fancyhdr,fancybox}
\usepackage{indentfirst}
\usepackage{verbatim}
\usepackage[sort&compress, numbers]{natbib}
\usepackage{geometry}
\usepackage{extarrows} 
\usepackage{color}
\usepackage[small]{caption2}
\usepackage{microtype}
\DisableLigatures[f]{encoding = *, family = *}

\usepackage{times}     

\usepackage{amssymb}                 
\usepackage{amsthm}
\usepackage{mathrsfs}                
\usepackage{bbm}
\usepackage{hyperref}  

\newcommand{\paperfont}{\fontsize{11pt}{1.2\baselineskip}\selectfont}
\pagebreak[4]
\begin{document}

\theoremstyle{definition}
\makeatletter
\thm@headfont{\bf}
\makeatother
\newtheorem{theorem}{Theorem}[section]
\newtheorem{definition}[theorem]{Definition}
\newtheorem{lemma}[theorem]{Lemma}
\newtheorem{proposition}[theorem]{Proposition}
\newtheorem{corollary}[theorem]{Corollary}
\newtheorem{remark}[theorem]{Remark}
\newtheorem{example}[theorem]{Example}
\newtheorem{assumption}[theorem]{Assumption}

\lhead{}
\rhead{}
\lfoot{}
\rfoot{}

\renewcommand{\refname}{References}
\renewcommand{\figurename}{Figure}
\renewcommand{\tablename}{Table}
\renewcommand{\proofname}{Proof}

\newcommand{\diag}{\mathrm{diag}}
\newcommand{\tr}{\mathrm{tr}}
\newcommand{\re}{\mathrm{Re}}
\newcommand{\one}{\mathbbm{1}}
\newcommand{\loc}{\textrm{loc}}

\newcommand{\Pnum}{\mathbb{P}}
\newcommand{\Enum}{\mathbb{E}}
\newcommand{\Rnum}{\mathbb{R}}
\newcommand{\Cnum}{\mathbb{C}}
\newcommand{\Znum}{\mathbb{Z}}
\newcommand{\Nnum}{\mathbb{N}}
\newcommand{\abs}[1]{\left\vert#1\right\vert}
\newcommand{\set}[1]{\left\{#1\right\}}
\newcommand{\norm}[1]{\left\Vert#1\right\Vert}
\newcommand{\innp}[1]{\langle {#1}\rangle}
\newcommand{\style}{\setlength{\itemsep}{1pt}\setlength{\parsep}{1pt}\setlength{\parskip}{1pt}}

\title{\textbf{Limit theorems for generalized density-dependent Markov chains and bursty stochastic gene regulatory networks}}
\author{Xian Chen$^1$,\;\;\;Chen Jia$^{2,3,*}$ \\
\footnotesize $^1$ School of Mathematical Sciences, Xiamen University, Xiamen, 361005, China.\\
\footnotesize $^2$ Beijing Computational Science Research Center, Beijing 100193, China. \\
\footnotesize $^3$ Department of Mathematics, Wayne State University, Michigan, Detroit 48202, U.S.A.\\
\footnotesize $^*$ Correspondence: chenjia@wayne.edu}
\date{}                              
\maketitle                           
\thispagestyle{empty}                

\paperfont

\begin{abstract}
Stochastic gene regulatory networks with bursting dynamics can be modeled mesocopically as a generalized density-dependent Markov chain (GDDMC) or macroscopically as a piecewise-deterministic Markov process (PDMP). Here we prove a limit theorem showing that each family of GDDMCs will converge to a PDMP as the system size tends to infinity. Moreover, under a simple dissipative condition, we prove the existence and uniqueness of the stationary distribution and the exponential ergodicity for the PDMP limit via the coupling method. Further extensions and applications to single-cell stochastic gene expression kinetics and bursty stochastic gene regulatory networks are also discussed and the convergence of the stationary distribution of the GDDMC model to that of the PDMP model is also proved.\\

\end{abstract}

\section{Introduction}
Density-dependent Markov chains (DDMCs) have been widely applied to model various stochastic systems in chemistry, ecology, and epidemics \cite{ethier2009markov, anderson2015stochastic}. In particular, they serve as a fundamental dynamic model for stochastic chemical reactions. If a chemical reaction system is well mixed and the numbers of molecules are very large, random fluctuations can be ignored and the evolution of the concentrations of all chemical species can be modeled macroscopically as a set of deterministic ordinary differential equations (ODEs) based on the law of mass action, dating back to the 18th century. If the numbers of participating molecules are not large, however, random fluctuations can no longer be ignored and the evolution of the system is usually modeled mesocopically as a DDMC. The Kolmogorov backward equation of the DDMC model turns out to be the famous chemical master equation, which is first introduced by Leontovich \cite{leontovich1935basic} and Delbr\"{u}ck \cite{delbruck1940statistical}. At the center of the mesoscopic theory of chemical reaction kinetics is a limit theorem proved by Kurtz in the 1970s \cite{kurtz1971limit, kurtz1972relationship, kurtz1976limit, kurtz1978strong}, which states that when the volume of the reaction vessel tends to infinity, the trajectory of the mesoscopic DDMC model will converge to that of the macroscopic ODE model (in probbaility \cite{kurtz1972relationship} or almost surely \cite{kurtz1978strong}) on any finite time interval, whenever the initial value converges. This limit theorem interlinks the deterministic and stochastic descriptions of chemical reaction systems and establishes a rigorous mathematical foundation for the nowadays widely used DDMC models.

The situation becomes more complicated when it comes to the stochastic biochemical reaction kinetics underlying single-cell gene expression and, more generally, gene regulatory networks. One reason of complexity is that biochemical reactions involved in gene expression usually possess multiple different time scales, spanning many orders of magnitude \cite{moran2013sizing}. Another source of complexity is the small copy numbers of participating molecules: there is usually only one copy of DNA on which a gene is located, mRNAs can be equally rare, and most proteins are present in less than 100 copies per bacterial cell \cite{paulsson2005models}. Over the past two decades, numerous single-cell experiments \cite{cai2006stochastic, suter2011mammalian} have shown that the synthesis of many mRNAs and proteins in an individual cell may occur in random bursts | short periods of high expression intensity followed by long periods of low expression intensity. To describe the experimentally observed bursting kinetics, some authors \cite{friedman2006linking, pajaro2015shaping, jkedrak2016time, bressloff2017stochastic, jia2017emergent, jia2019macroscopic} have modeled gene expression kinetics as a piecewise deterministic Markov process (PDMP) with discontinuous trajectories, where the jumps in the trajectories correspond to random transcriptional or translational bursts. On the other hand, some authors \cite{paulsson2000random, mackey2013dynamic, kumar2014exact, jia2017simplification, jia2017stochastic} have used the mesoscopic model of generalized density-dependent Markov chains (GDDMCs) to describe the molecular mechanism underlying stochastic gene expression. This raises the important question of whether the macroscopic PDMP model can be viewed as the limit of the mesoscopic GDDMC model.

In this paper, we introduce a family of stochastic processes called GDDMCs, which generalize the classical DDMCs and have important biological significance. Furthermore, we prove a functional limit theorem for GDDMCs using the theory of martingale problems. In particular, we show that the limit process of each family of GDDMCs is a PDMP with a L\'{e}vy-type generator. This limit theorem, in analogy to the pioneering work of Kurtz, interlinks the macroscopic and mesoscopic descriptions of stochastic gene regulatory network with bursting dynamics and establishes a rigorous mathematical foundation for the empirical PDMP models.

Another important biological problem is to study the stationary distribution for stochastic gene expression. In this simplest case that the gene of interest is unregulated, the stationary distributions of the mesoscopic GDDMC model and the macroscopic PDMP model turn out to be a negative binomial distribution \cite{shahrezaei2008analytical} and a gamma distribution \cite{friedman2006linking}, respectively. Both the two distributions fit single-cell data reasonably well \cite{cai2006stochastic}. Therefore, it is natural to ask when the stationary distribution exists and is unique for the two models and whether the stationary distribution of the GDDMC model will converge to that of the PDMP model. In this paper, we prove the existence and uniqueness of the stationary distribution for the two models under a simple dissipative condition. Under the same condition, we also prove the convergence of the stationary distribution of the GDDMC model to that of the PDMP model.

From the mathematical aspect, another interesting question to study is whether the limit process is ergodic, that is, whether the time-dependent distribution of the PDMP limit will converge to its stationary distribution. In previous studies \cite{mackey2008dynamics, mackey2013dynamic, mackey2016simple}, Mackey et al. have shown that if the stationary distribution exists, then the PDMP model is ergodic in some sense. In this paper, using the coupling method, we reinforce this result by showing that the PDMP limit is actually exponentially ergodic under a simple dissipative condition, that is, the time-dependent distribution will converge to the stationary distribution at an exponential speed.

As another biological application, we propose a mesoscopic GDDMC model of bursty stochastic gene regulatory networks with multiple genes, complex burst-size distributions, and complex network topology. Then our abstract limit theorem is applied to investigate the macroscopic PDMP limit of the mesoscopic GDDMC model.

The structure of this article is organized as follows. In Section 2, we give the rigorous definition of a family of GDDMCs and construct the trajectories of its PDMP limit. In Section 3, we state four main theorems. In Section 4, we apply our abstract theorems to the specific biological problem of single-cell stochastic gene expression and obtain some further mathematical results. In Section 5, we apply the limit theorem to study the macroscopic limit of a complex stochastic regulatory network with bursting dynamics. The remaining sections are devoted to the detailed proofs of the main theorems.

\section{Model}\label{secmodel}
In recent years, there has been a growing attention to gene regulatory networks and biochemical reaction networks modeled by a GDDMC, which generalizes the classical DDMC \cite{kurtz1971limit, kurtz1972relationship, kurtz1976limit, kurtz1978strong}. In this paper, we consider a family of continuous-time Markov chains $X_V = \set{X_V(t):t\geq 0}$ on the $d$-dimensional lattice
\begin{equation*}
E_V = \set{\tfrac{n}{V}:n=(n_1,n_2,\ldots,n_d)\in\Nnum^d}
\end{equation*}
with transition rate matrix $Q_V = (q_V(x,y))$, where $\Nnum$ is the set of nonnegative integers and $V>0$ is a scaling parameter. Such Markov chains have been widely applied to model the evolution of the concentrations of multiple chemical species undergoing stochastic chemical reactions \cite{anderson2015stochastic}. Specifically, for each $1\leq i\leq d$, $n_i$ stands for the copy number of the $i$th chemical species and $V$ usually stands for the size of the system \cite{kurtz1972relationship}. Then $n_i/V$ represents the concentration of the $i$th chemical species.

The transition rates of this Markov chain consist of two parts:
\begin{equation*}
q_V(x,y) = \hat{q}_V(x,y)+\tilde{q}_V(x,y),\;\;\;x,y\in E_V,x\neq y,
\end{equation*}
where $\hat{q}_V(x,y)$ is called the \emph{reaction part} and $\tilde{q}_V(x,y)$ is called the \emph{bursting part}. The functional forms of the two parts are described as follows. For each $m\in \Znum^d-\{0\}$, we assume that there exists a locally bounded function $\beta_m:\Rnum_+^d\rightarrow\Rnum_+$ such that
\begin{equation}\label{reaction}
\hat{q}_V\left(\tfrac{n}{V},\tfrac{n+m}{V}\right) = V\beta_m\left(\tfrac{n}{V}\right),\;\;\;n\in \Nnum^d,n+m\in \Nnum^d,
\end{equation}
where $\Znum$ is the set of integers and $\Rnum_{+}$ is the set of nonnegative real numbers. Throughout this paper, we assume that
\begin{equation*}
\sum_{m\neq 0}|m|\beta_m(x)<\infty,\;\;\;\textrm{for any}\;x\in\Rnum_+^d.
\end{equation*}
In fact, the condition \eqref{reaction} can be relaxed slightly as
\begin{equation}\label{reactiongeneral}
\lim_{V\rightarrow\infty}\sup_{n\leq kV}\Big|\tfrac{1}{V}\hat{q}_V\left(\tfrac{n}{V},\tfrac{n+m}{V}\right)
-\beta_m\left(\tfrac{n}{V}\right)\Big| = 0,\;\;\textrm{for any}\;k>0.
\end{equation}
Moreover, we assume that there exists a positive integer $N$ such that
\begin{equation}\label{bursting}
\tilde{q}_V\left(\tfrac{n}{V},\tfrac{n+m}{V}\right)
= \sum_{i=1}^Nc_i\left(\tfrac{n}{V}\right)\mu_i\left[\tfrac{m}{V},\tfrac{m+1}{V}\right),
\;\;\;n\in \Nnum^d, m\in\Nnum^d-\{0\},
\end{equation}
where for each $1\leq i\leq N$, $c_i:\Rnum_+^d\rightarrow\Rnum_+$ is a Lipschitz function with Lipschitz constant $L_{c_i}>0$, $\mu_i$ is a Borel probability measure on $\Rnum_{+}^d-\{0\}$ with finite mean, and
\begin{equation*}
\left[\tfrac{m}{V},\tfrac{m+1}{V}\right) \triangleq
\left[\tfrac{m_1}{V},\tfrac{m_1+1}{V}\right)\times\ldots\times\left[\tfrac{m_d}{V},\tfrac{m_d+1}{V}\right).
\end{equation*}
Similarly, the condition \eqref{bursting} can be relaxed slightly as
\begin{equation*}
\tilde{q}_V\left(\tfrac{n}{V},\tfrac{n+m}{V}\right) = \sum_{i=1}^Nc_i\left(\tfrac{n}{V}\right)p_i(V,m),\;\;\;
n\in\Nnum^d, m\in\Nnum^d-\{0\},
\end{equation*}
where $p_i(V,m)$ satisfies the following three conditions for each $1\leq i\leq N$:
\begin{equation}\label{conditions}
\begin{split}
&\textrm{(a)}\; \sum_{m\in\Nnum^d-\{0\}}|m|p_i(V,m)<\infty,\;\;\;\textrm{for any}\;V>0,\\
&\textrm{(b)}\; \lim_{V\rightarrow\infty}\sum_{m\in\Nnum^d-\{0\}}p_i(V,m) = 1,\\
&\textrm{(c)}\; \lim_{V\rightarrow\infty}V^d\sup_{|m|\leq kV}
\left|p_i(V,m)-\mu_i\left[\tfrac{m}{V},\tfrac{m+1}{V}\right)\right| = 0,\;\; \textrm{for any}\;k>0.
\end{split}
\end{equation}
We shall refer to $X_V$ as a $d$-dimensional GDDMC. If the transition rates only contain the reaction part, then $X_V$ reduces to the classical DDMC \cite{ethier2009markov, anderson2015stochastic}.

\begin{remark}\label{relax}
In fact, the condition (c) in \eqref{conditions} can be further relaxed. If the term $p_i(V,m)$ is concentrated on a $\tilde{d}$-dimensional hyperplane $H$ with $\tilde{d}<d$, that is,
\begin{equation*}
p_i(V,m) = 0,\;\;\;\textrm{whenever}\;m\notin H,
\end{equation*}
then $\mu_i$ is a probability measure concentrated on $H$ and the condition (c) in \eqref{conditions} can be relaxed  with $d$ replaced by $\tilde{d}$.
\end{remark}

\begin{remark}
If we use a GDDMC to model the expression levels of a family of proteins in a stochastic gene regulatory network, then the positive integer $N$ is usually chosen as the number of genes in the network. Moreover, the function $c_i$ describes the transcription rate of the $i$th gene and the probability measure $\mu_i$ or $p_i(V,\cdot)$ represents the burst-size distribution of the $i$th protein. These biological concepts will be explained in more detail in Sections 4 and 5.
\end{remark}

\begin{remark}
Suppose that a chemical reaction system contains the reaction
\begin{equation*}
a_1S_1+a_2S_2+\cdots+a_dS_d \rightarrow b_1S_1+b_2S_2+\cdots+b_dS_d,
\end{equation*}
where $S_1,S_2,\cdots,S_d$ are all chemical species involved in the chemical reaction system and $a_i$ and $b_i$ are nonnegative integers for each $1\leq i\leq d$. In this case, the GDDMC model of the chemical reaction system has a transition from $n/V$ to $(n+m)/V$ with $m = (b_1-a_1,b_2-a_2,\cdots,b_d-a_d)$. The corresponding transition rate from $n/V$ to $(n+m)/V$ has the form of
\begin{equation}\label{combinatorial}
\hat{q}_V\left(\tfrac{n}{V},\tfrac{n+m}{V}\right)
= \frac{k}{V^{a_1+\cdots+a_d-1}}C_{n_1}^{a_1}\cdots C_{n_d}^{a_d},
\end{equation}
where $k$ is the rate constant of the reaction. Moreover, it is easy to check that the condition \eqref{reactiongeneral} holds with $\beta_m$ being the polynomial
\begin{equation*}
\beta_m(x) = \frac{k}{a_1!\cdots a_d!}x_1^{a_1}\cdots x_d^{a_d}.
\end{equation*}
A DDMC model of a chemical reaction system with transition rates having the mass action kinetics \eqref{combinatorial} is often referred to as a Delbruck-Gillespie process \cite{jia2017emergent}.
\end{remark}

Our major aim is to study the limit behavior of $X_V$ as the scaling parameter $V\rightarrow\infty$. In fact, the limit process of $X_V$ turns out to be a PDMP with discontinuous trajectories, which can be constructed as follows. Let $F:\Rnum_+^d\rightarrow\Rnum^d$ be a vector field defined by
\begin{equation*}
F(x) = \sum_{m\neq 0}m\beta_m(x).
\end{equation*}
We assume that $F$ is a Lipschitz function with Lipschitz constant $L_F>0$. For the Markov chain $X_V$, since the transitions from the first orthant $\Rnum_+^d$ to other orthants $\Rnum^d-\Rnum_+^d$ are forbidden, it is easy to see that $\beta_m(x) = 0$ if $x_i = 0$ and $m_i<0$ for some $1\leq i\leq d$.
Therefore, for any $x\in\Rnum_+^d$ with $x_i = 0$, we have
\begin{equation*}
F_i(x) = \sum_{m\neq 0}m_i\beta_m(x) \geq 0.
\end{equation*}
This shows that on the boundary of the first orthant, the vector field $F$ points towards the interior of the first orthant. Thus, the ordinary differential equation $\dot x = F(x)$ has a global flow $\phi: \Rnum_+\times \Rnum_+^d\rightarrow \Rnum_+^d$ satisfying
\begin{equation}\label{flow}
\frac{d}{dt}\phi(t,x) = F(\phi(t,x)),\;\;\;\phi(0,x) = x.
\end{equation}
The limit process $X = \{X(t):t\geq 0\}$ can be constructed as follows. Set
\begin{equation*}
c(x) = \sum_{i=1}^Nc_i(x),\;\;\;\tilde{c}_i(x) = \frac{c_i(x)}{c(x)},\;\;\;1\leq i\leq N,
\end{equation*}
where we define $0/0 = 1$. Suppose that $X_0 = x\in \Rnum_+^d$. First, we selection a jump time $T_1$ with survival function
\begin{equation*}
\Pnum(T_1>t) = e^{-\int_0^tc(\phi(s,x))ds}.
\end{equation*}
Next, we select a random vector $Z_1$ with distribution
\begin{equation*}
\Pnum(Z_1\in\cdot|T_1) = \sum_{i=1}^N\tilde{c}_i(\phi(T_1,x))\mu_i(\cdot).
\end{equation*}
Then the trajectory of $X$ before $T_1$ is constructed by
\begin{equation*}
X(t) = \begin{cases}
\phi(t,x), &0\leq t<T_1,\\
\phi(T_1,x)+Z_1, &t=T_1.
\end{cases}
\end{equation*}
Repeating this procedure, for some integer $n\geq 1$, suppose that the trajectory of $X$ before the jump time $T_n$ has been constructed. Then we independently select the next inter-jump time $T_{n+1}-T_n$ with survival function
\begin{equation*}
\Pnum(T_{n+1}-T_n>t|X(T_n)) = e^{-\int_0^t c(\phi(s,X(T_n))ds}.
\end{equation*}
Next, we independently select a random vector $Z_{n+1}$ with distribution
\begin{equation}\label{jumpvector}
\Pnum(Z_{n+1}\in\cdot|X(T_n),T_n,T_{n+1}) = \sum_{i=1}^N\tilde{c}_i(\phi(T_{n+1}-T_n,X(T_n)))\mu_i(\cdot).
\end{equation}
Then the trajectory of $X$ between $T_n$ and $T_{n+1}$ is constructed by
\begin{equation}\label{induction}
X(t) = \begin{cases}
\phi(t-T_n,X(T_n)), &T_n\leq t<T_{n+1},\\
\phi(T_{n+1}-T_n,X(T_n))+Z_{n+1}, &t=T_{n+1}.
\end{cases}
\end{equation}
Moreover, we assume that $X$ enters the tomb state $\Delta = \infty$ after the explosion time
\begin{equation*}
T_{\infty}=\lim_{n\rightarrow \infty}T_n.
\end{equation*}
In this way, we obtain a Markov process $X$, which is widely known as a PDMP \cite{davis1984piecewise}.

\section{Results}\label{secresult}
Before stating our results, we introduce some notation. Let $S$ be a metric space and let $\mathcal{P}(S)$ denote the set of Borel probability measures on $S$. In this paper, the following five function spaces will be frequently used. Let $B(S)$ denote the space of bounded Borel measurable functions on $S$. Let $C_b(S)$ denote the space of bounded continuous functions on $S$. Let $C_c(S)$ denote the space of continuous functions on $S$ with compact supports. Let $C_0(S)$ denote the space of continuous functions on $S$ vanishing at infinity. Let $D(\Rnum_+,S)$ denote the space of c\`{a}dl\`{a}g functions $f:\Rnum_+\rightarrow S$ endowed with the Skorohod topology.

We next recall an important definition \cite[Section 4.2]{ethier2009markov}.

\begin{definition}\label{martingale}
Let $S$ be a metric space and let $\mathcal{R}$ be a linear operator on $B(S)$ with domain $\mathcal{D}(\mathcal{R})$. Let $Y = \set{Y(t):t\geq 0}$ be a stochastic process with sample paths in $D(\Rnum_+,S)$. We say that $Y$ is a solution to the \emph{martingale problem} for $\mathcal{R}$ if for any $f\in \mathcal{D}(\mathcal{R})$,
\begin{equation*}
f(Y(t))-f(Y(0))-\int_0^t\mathcal{R}f(Y(s))ds
\end{equation*}
is a martingale with respect to the natural filtration generated by $Y$. For any $\nu\in\mathcal{P}(S)$, we say that $Y$ is a solution to the martingale problem for $(\mathcal{R},\nu)$ if $Y$ is a solution to martingale problem for $\mathcal{R}$ and $Y$ has the initial distribution $\nu$. The solution to the martingale problem for $(\mathcal{R},\nu)$ is said to be \emph{unique} if any two solutions have the same finite-dimensional distributions. The martingale problem for $(\mathcal{R},\nu)$ is said to be \emph{well posed} if its solution exists and is unique.
\end{definition}

For any $V>0$, let $\mathcal{A}_V$ be a linear operator on $B(E_V)$ with domain $\mathcal{D}(\mathcal{A}_V) = C_c(E_V)$ defined by
\begin{equation*}
\begin{split}
\mathcal{A}_Vf\left(\tfrac{n}{V}\right)
=&\; \sum_{m\neq 0}V\beta_m\left(\tfrac{n}{V}\right)
\left[f\left(\tfrac{n+m}{V}\right)-f\left(\tfrac{n}{V}\right)\right] \\
&\; +\sum_{i=1}^Nc_i\left(\tfrac{n}{V}\right)\sum_{m\in\Nnum^d}p_i(V,m)
\left[f\left(\tfrac{n+m}{V}\right)-f\left(\tfrac{n}{V}\right)\right].
\end{split}
\end{equation*}
In the special case of $N = 1$, the above operator reduces to
\begin{equation}\label{eqdisoperator}
\begin{split}
\mathcal{A}_Vf\left(\tfrac{n}{V}\right)
=&\; \sum_{m\neq 0}V\beta_m\left(\tfrac{n}{V}\right)
\left[f\left(\tfrac{n+m}{V}\right)-f\left(\tfrac{n}{V}\right)\right] \\
&\; +c\left(\tfrac{n}{V}\right)\sum_{m\in\Nnum^d}p(V,m)
\left[f\left(\tfrac{n+m}{V}\right)-f\left(\tfrac{n}{V}\right)\right].
\end{split}
\end{equation}
The following theorem characterizes $X_V$ from the perspective of martingale problems.

\begin{theorem}\label{MC}
Let $\nu_V$ be the initial distribution of $X_V$. Then $X_V$ is the unique solution to the martingale problem for $(\mathcal{A}_V,\nu_V)$ with sample paths in $D(\Rnum_+,E_V)$. In particular, $X_V$ is nonexplosive.
\end{theorem}

\begin{proof}
The proof of this theorem will be given at the end of Section 6.
\end{proof}

Furthermore, let $\mathcal{A}$ be a L\'{e}vy-type operator on $B(\Rnum_+^d)$ with domain $\mathcal{D}(\mathcal{A}) = C_c^1(\Rnum_+^d)$ defined by
\begin{equation*}
\mathcal{A}f(x)
= \sum_{i=1}^dF_i(x)\partial_if(x)+\sum_{i=1}^Nc_i(x)\int_{\Rnum_+^d}[f(x+y)-f(x)]\mu_i(dy).
\end{equation*}
In the special case of $N = 1$, the above operator reduces to
\begin{equation}\label{eqoperator}
\mathcal{A}f(x)
= \sum_{i=1}^dF_i(x)\partial_if(x)+c(x)\int_{\Rnum_+^d}[f(x+y)-f(x)]\mu(dy).
\end{equation}
This L\'{e}vy-type operator is degenerate in the sense that it has no diffusion term. If fact, the existence and uniqueness of the martingale problem for a non-degenerate L\'{e}vy-type operator with a bounded $c(x)$ has been proved by Stroock \cite{stroock1975diffusion}. However, this result cannot be applied to a degenerate L\'{e}vy-type operator with an unbounded $c(x)$.

Let $N_t = \sup\{n\geq 1:T_n\leq t\}$ be the number of jumps of $X$ by time $t$. In fact, the classical theory of PDMPs relies on the basic assumption that $\Enum N_t<\infty$ for any $t\geq 0$, which guarantees $X$ to be nonexplosive. Under this assumption, Davis \cite{davis1984piecewise} has used the theory of multivariate point processes to find the extended generator of $X$. However, this assumption may not be true under our current framework. The following theorem characterizes $X$ from the perspective of martingale problems and provides a simple criterion for the nonexplosiveness of $X$.

\begin{theorem}\label{thmwellposed}
Let $\nu$ be the initial distribution of $X$. Then $X$ is the unique solution to the martingale problem for $(\mathcal{A},\nu)$ with sample paths in $D(\Rnum_+,\Rnum_+^d)$. In particular, $X$ is nonexplosive.
\end{theorem}

\begin{proof}
The proof of this theorem will be given in Section 6.
\end{proof}


For any two probability measures $\mu_1,\mu_2\in\mathcal{P}(\Rnum_{+}^d)$ with finite means, recall that the $L^1$-Wasserstein distance between them is defined as
\begin{equation*}
W(\mu_1,\mu_2) = \inf_{\gamma\in G(\mu_1,\mu_2)}\int_{\Rnum_{+}^d\times\Rnum_{+}^d}|x-y|d\gamma(x,y),
\end{equation*}
where $G(\mu_1,\mu_2)$ is the collection of Borel probability measures on $\Rnum_{+}^d\times\Rnum_{+}^d$ with marginals $\mu_1$ and $\mu_2$ on the first and second factors, respectively \cite{chen2004markov}. The following theorem characterizes the exponential ergodicity of $X$ under the $L^1$-Wasserstein distance.

\begin{theorem}\label{thmergodicity}
Suppose that there exists
\begin{equation*}
r > \sum_{i=1}^NL_{c_i}\int_{\Rnum_+^d}|x|\mu_i(dx)
\end{equation*}
such that the following dissipative condition holds:
\begin{equation}\label{eqass}
\langle F(x)-F(y), x-y\rangle \leq -r|x-y|^2,\;\;\;\textrm{for any}\;x,y\in\Rnum_+^d,
\end{equation}
where $\langle\cdot,\cdot\rangle$ denotes the standard inner product on $\Rnum^d$. Then $X$ has a unique stationary distribution $\pi$ with finite mean such that
\begin{equation*}
W(\pi_t,\pi)\leq W(\pi_0,\pi)e^{-\tilde{r}t},\;\;\;\textrm{for any}\;t\geq 0,
\end{equation*}
where $\pi_t$ is the distribution of $X(t)$ and
\begin{equation*}
\tilde{r} = r-\sum_{i=1}^NL_{c_i}\int_{\Rnum_+^d}x\mu_i(dx) > 0.
\end{equation*}
\end{theorem}

\begin{proof}
The proof of this theorem will be given in Section 7.
\end{proof}

In a previous study \cite{mackey2013dynamic}, the authors have shown that if the stationary distribution of $X$ exists, then it is ergodic in some sense. In this paper, we reinforce this result by showing that $X$ is actually exponentially ergodic under a simple dissipative condition.

Let $\{\mu_V:V>0\}$ be a sequence of probability measures on a measurable space $S$ and let $\mu$ be a probability measure on $S$. In the following, we shall use the symbol $\mu_V\Rightarrow\mu$ to denote the weak convergence of $\mu_V$ to $\mu$ as $V\rightarrow\infty$. Since the Markov chain $X_V$ has c\`{a}dl\`{a}g trajectories, its distribution is a probability measure $\mu_V$ on the path space $D(\Rnum_+,\Rnum_+^d)$ defined by
\begin{equation*}
\mu_V(\cdot) = \Pnum(X_V\in\cdot).
\end{equation*}
Let $Y$ be another process with sample paths in $D(\Rnum_+,\Rnum_+^d)$ and let $\mu$ be the distribution of $Y$. We say that $X_V$ converges weakly to $Y$ in $D(\Rnum_+,\Rnum_+^d)$, denoted by $X_V\Rightarrow Y$, if $\mu_V\Rightarrow\mu$ in $D(\Rnum_+,\Rnum_+^d)$. The following theorem characterizes the limit behavior of $X_V$ as $V\rightarrow\infty$.

\begin{theorem}\label{convergence}
Suppose that $\beta_m$ is nonzero for a finite number of $m$. Let $\nu_{V}$ be the initial distribution of $X_V$ and let $\nu$ be the initial distribution of $X$. If $\nu_V\Rightarrow\nu$ as $V\rightarrow\infty$, then $X_V\Rightarrow X$ in $D(\Rnum_+,\Rnum_+^d)$ as $V\rightarrow\infty$.
\end{theorem}

\begin{proof}
The proof of this theorem will be given in Section 8.
\end{proof}

\section{Applications in single-cell stochastic gene expression}
In this section, we apply our abstract theorems to an important biological problem. Over the past two decades, significant progress has been made in the kinetic theory of single-cell stochastic gene expression \cite{paulsson2005models}. Based on the central dogma of molecular biology, the expression of a gene in a single cell with size $V$ can be described by a standard two-stage model \cite{shahrezaei2008analytical} consisting of transcription and translation, as illustrated in Fig. \ref{model}(a). The transcription and translation steps describe the synthesis of the mRNA and protein, respectively. Both the mRNA and protein can be degraded. Here, $s_n$ is the transcription rate, $u$ is the translation rate, and $v$ and $r$ are the degradation rates of the mRNA and protein, respectively. In real biological systems, the products of many genes may directly or indirectly regulate their own expression via a positive or negative feedback loop. Due to feedback controls, the transcription rate $s_n = c(n/V)$ is a function of the protein concentration $n/V$. In the presence of a positive feedback loop, $c(x)$ is an increasing function. In the presence of a negative feedback loop, $c(x)$ is a decreasing function. If the gene is unregulated, $c(x) = c$ is a constant function.
\begin{figure}[!htb]
\centerline{\includegraphics[width=1\textwidth]{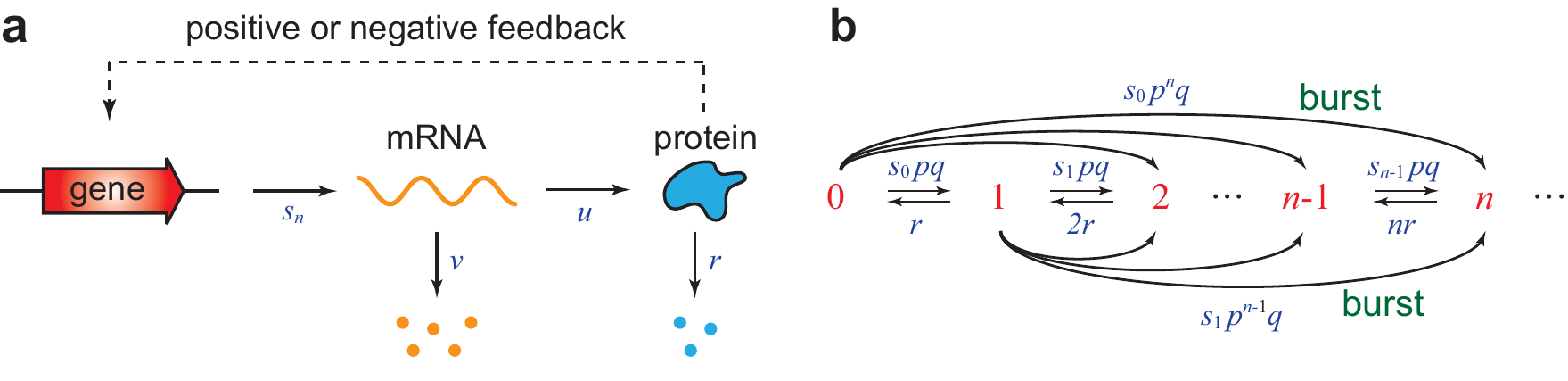}}
\caption{\textbf{Models of single-cell stochastic gene expression.} (a) Central dogma of molecular biology. (b) The transition diagram of the Markov chain model.}\label{model}
\end{figure}

In single-cell experiments \cite{suter2011mammalian}, it was consistently observed that the mRNA decays much faster than the corresponding protein \cite{shahrezaei2008analytical}. This suggests that the process of protein synthesis followed by mRNA degradation is essentially instantaneous. Once an mRNA copy is synthesized, it can either produce a protein copy with probability $p = u/(u+v)$ or be degraded with probability $q = v/(u+v)$. Therefore, the probability that each mRNA copy produces $k$ protein copies before it is finally degraded is $p^kq$, which has a geometric distribution. Then the rate at which $k$ protein copies are synthesized will be the product of the transcription rate $s_n$ and the geometric probability $p^kq$. Thus, the evolution of the protein copy number in a single cell can be modeled by a continuous-time Markov chain $N = \{N(t):t\geq 0\}$ on nonnegative integers with transition diagram depicted in Fig. \ref{model}(b) \cite{paulsson2000random, jia2017emergent}. The phenomenon that a large number of protein copies can be produced within a very short period is referred to as random translational bursts, which correspond to the long-range jumps in Fig. \ref{model}(b) \cite{jia2017simplification}. The number of protein copies synthesized in a single burst is called the \emph{burst size} of the protein. Since the burst size has a geometric distribution, its expected value is given by
\begin{equation*}
\sum_{k=1}^\infty kp^kq = \frac{p}{q}.
\end{equation*}

In many single-cell experiments such as flow cytometry and fluorescence microscopy, one usually obtains data of protein concentrations, instead of protein copy numbers \cite{cai2006stochastic}. Let $V>0$ be a scaling parameter which usually denotes the average cell volume \cite{taniguchi2010quantifying} or maximal protein copy number \cite{assaf2011determining, lv2014constructing}, and let $X_V(t) = N(t)/V$ denote the concentration of the protein at time $t$. Then the concentration process $X_V = \{X_V(t):t\geq 0\}$ is a one-dimensional GDDMC on the lattice
\begin{equation*}
E_V = \set{\tfrac{n}{V}:n=0,1,2,\cdots}
\end{equation*}
associated with the operator
\begin{equation}\label{defgenechain}
\mathcal{A}_Vf\left(\tfrac{n}{V}\right)
= rn\left[f\left(\tfrac{n-1}{V}\right)-f\left(\tfrac{n}{V}\right)\right]
+c\left(\tfrac{n}{V}\right)\sum_{m=1}^{\infty}p_V^{m}q_V
\left[f\left(\tfrac{n+m}{V}\right)-f\left(\tfrac{n}{V}\right)\right].
\end{equation}
Here we assume that $p = p_V$ and $q = q_V$ depend on $V$ and $c:\Rnum_+\rightarrow\Rnum_+$ is a Lipschitz function. It is easy to see that $\mathcal{A}_V$ is a special case of the operator \eqref{eqdisoperator} with
\begin{equation*}
\beta_{-1}(x) = rx,\;\;\;\beta_m(x) = 0\;\textrm{for any}\;m\neq -1,\;\;\;p(V,m) = p_V^mq_V.
\end{equation*}
In living cells, the mean burst size $p_V/q_V$ of the protein is large, typically on the order of 100 for a bacterial gene \cite{paulsson2005models}. Thus, it is natural to require that the mean burst size scales with the parameter $V$ as
\begin{equation*}
\frac{p_V}{q_V} = \frac{V}{\lambda},
\end{equation*}
where $\lambda>0$ is a constant. On the other hand, let $X = \{X(t):t\geq 0\}$ be a PDMP associated with the operator
\begin{equation}\label{deflevygene}
\mathcal{A}f(x) = -rxf'(x)+c(x)\int_0^{\infty}[f(x+y)-f(x)]\lambda e^{-\lambda y}dy.
\end{equation}
It is worth noting that $\mathcal{A}$ is a special case of the operator \eqref{eqoperator} with
\begin{equation*}
F(x) = -rx,\;\;\;\mu(dx)= \lambda e^{-\lambda x}dx.
\end{equation*}

The following theorem, which follows directly from Theorem \ref{convergence}, characterizes the limit behavior of the concentration process $X_V$ as $V\rightarrow\infty$.

\begin{theorem}
Let $\nu_{V}$ be the initial distribution of the GDDMC model $X_V$ of single-cell stochastic gene expression kinetics and let $\nu$ be the initial distribution of the PDMP model $X$. If $\nu_V\Rightarrow\nu$ as $V\rightarrow\infty$, then $X_V\Rightarrow X$ in $D(\Rnum_+,\Rnum_+)$ as $V\rightarrow\infty$.
\end{theorem}

\begin{proof}
By Theorem \ref{convergence}, we only need to check that $p(V,m) = p_V^mq_V$ satisfies the three conditions listed in \eqref{conditions}. For any $V>0$, it is easy to see that
\begin{equation*}
\sum_{m=1}^\infty mp(V,m) = \sum_{m=1}^\infty mp_V^mq_V = \frac{p_V}{q_V} < \infty.
\end{equation*}
Since $p_V/q_V = V/\lambda$ and $q_V = 1-p_V$, we have $p_V\rightarrow 1$ as $V\rightarrow\infty$. This shows that
\begin{equation*}
\lim_{V\rightarrow\infty}\sum_{m=1}^{\infty}p(V,m)
= \lim_{V\rightarrow\infty}\sum_{m=1}^{\infty}p_V^mq_V = \lim_{V\rightarrow\infty}p_V = 1.
\end{equation*}
Finally, it follows from the mean value theorem that
\begin{equation*}
\mu\left[\tfrac{m}{V},\tfrac{m+1}{V}\right)
= \int_{\frac{m}{V}}^{\frac{m+1}{V}}\lambda e^{-\lambda x}dx
= \frac{\lambda}{V}e^{-\lambda\xi_m},
\end{equation*}
where $\xi_m$ is between $m/V$ and $(m+1)/V$. Applying the mean value theorem again yields
\begin{align*}
V\left|p(V,m)-\mu\left[\tfrac{m}{V},\tfrac{m+1}{V}\right)\right|
&= |Vp_V^mq_V-\lambda e^{-\lambda\xi_m}| = |Vq_V e^{m\log(1-q_V)}-\lambda e^{-\lambda\xi_m}|\\
&\leq |Vq_V-\lambda|e^{m\log(1-q_V)}+\lambda|e^{m\log(1-q_V)}-e^{-\lambda\xi_m}|\\
&\leq |Vq_V-\lambda|+\lambda|m\log(1-q_V)+\lambda\xi_m|\\
&\leq |Vq_V-\lambda|
+\lambda\left|m\log(1-q_V)+\tfrac{\lambda m}{V}\right|+\lambda^2\left|\tfrac{m}{V}-\xi_m\right|\\
&\leq |Vq_V-\lambda|+\tfrac{\lambda m}{V}|V\log(1-q_V)+\lambda|+\tfrac{\lambda^2}{V}.
\end{align*}
Since $p_V/q_V = V/\lambda$ and $q_V = 1-p_V$, it is easy to check that
\begin{equation*}
\lim_{V\rightarrow\infty}Vq_V = -\lim_{V\rightarrow\infty}V\log(1-q_V) = \lambda.
\end{equation*}
Thus we finally obtain that
\begin{equation*}
\lim_{V\rightarrow\infty}V\sup_{0<m\leq kV}\left|p(V,m)-\mu\left[\tfrac{m}{V},\tfrac{m+1}{V}\right)\right|
= 0.
\end{equation*}
So far, we have validated all the three conditions listed in \eqref{conditions}.
\end{proof}

In fact, both the mesoscopic GDDMC model \cite{paulsson2000random, mackey2013dynamic, kumar2014exact, jia2017simplification, jia2017stochastic} and macroscopic PDMP model \cite{friedman2006linking, pajaro2015shaping, jkedrak2016time, bressloff2017stochastic, jia2017emergent, jia2019macroscopic} have been widely used to describe single-cell stochastic gene expression kinetics. In particular, the gene expression models described above are particular examples of the models studied in \cite{mackey2013dynamic}. In this paper, we establish a deep connection between the mesoscopic and macroscopic models by viewing the latter as the weak limit of the former in the Skorohod space. This provides a rigorous theoretical foundation and justifies the wide application for the empirical PDMP mdoel.

In our general theory, we have shown that if the dissipative condition \eqref{eqass} is satisfied, then there exists a unique stationary distribution for the limit process $X$ among all probability measures with finite means. However, for the PDMP model of stochastic gene expression, we can prove the stronger result that the stationary distribution is unique among all probability measures.

\begin{theorem}\label{invariant}
Suppose that $r > L_c/\lambda$ and $c(0)>0$. Then $X_V$ has a unique stationary distribution
\begin{equation}\label{invmeasure}
\pi_V\left(\tfrac{n}{V}\right)
= A_V\frac{p_V^n}{n!}\prod_{k=0}^{n-1}\left(\frac{1}{r}c\left(\tfrac{k}{V}\right)+k\right),\;\;\;n\geq 0,
\end{equation}
where $A_V>0$ is a normalization constant. Moreover, $X$ also has a unique stationary distribution $\pi(dx) = p(x)dx$, whose density is given by
\begin{equation*}
p(x) = Ax^{-1}e^{-\lambda x+\frac{1}{r}\int_1^x\frac{c(y)}{y}dy},\;\;\;x>0,
\end{equation*}
where $A>0$ is a normalization constant.
\end{theorem}

\begin{proof}
The fact that $\pi_V$ is a stationary distribution for $X_V$ follows from Corollary 3.3 in \cite{mackey2013dynamic} and the uniqueness of the stationary distribution follows from the irreducibility of $X_V$. When $c(0)>0$, any stationary distribution for $X$ must have a density \cite[Theorem 3.1]{lopker2013time} and thus its uniqueness follows from Corollary 4.9 in \cite{mackey2013dynamic}. The fact that $\pi$ is a stationary distribution for $X$ follows from Remark 4.10 in \cite{mackey2013dynamic}.
\end{proof}

\begin{remark}
In the degenerate case of $c(0) = 0$, state $0\in E_V$ is the only absorbing state of the Markov chain $X_V$ and thus $\pi_V = \delta_0$ is the unique stationary distribution for $X_V$, where $\delta_0$ denotes the point mass at $0$. Moreover, it is easy to see that $\pi = \delta_0$ is a stationary distribution for the limit process $X$, which has no density. By \cite[Theorem 2.2]{wang2010regularity} and \cite[Theorem 1]{komorowski2010ergodicity}, the stationary distribution of $X$ is unique if there exists $x_0\geq 0$ such that
\begin{equation*}
\liminf_{t\rightarrow\infty}\frac{1}{t}\int_0^t\Pnum_x(x_0-\delta< X_s< x_0+\delta)ds > 0,
\;\;\;\textrm{for any}\;x\geq 0,\delta>0.
\end{equation*}
Since $\pi = \delta_0$ has a finite mean, it follows from Theorem \ref{thmergodicity} that $W(\pi_t,\pi)\rightarrow 0$ as $t\rightarrow\infty$. Since convergence under the $L^1$- Wasserstein distance implies weak convergence, for any $x>0$ and $\delta>0$,
\begin{equation*}
\liminf_{t\rightarrow\infty}\frac{1}{t}\int_0^t\Pnum_x(X_s<\delta)ds
\geq \lim_{t\rightarrow\infty}\Pnum_x(X_t<\delta) = \pi([0,\delta)) = 1 > 0.
\end{equation*}
Therefore, $\pi = \delta_0$ is the unique stationary distribution for $X$.
\end{remark}

Recall that if the gene is unregulated, then $c(x) = c$ is a constant function. The following corollary follows directly from Theorem \ref{invariant}.

\begin{corollary}
Suppose that $c(x) = c>0$ is a constant function. Then the unique stationary distribution of $X_V$ is the negative binomial distribution
\begin{equation*}
\pi_V\left(\tfrac{n}{V}\right) = \frac{(c/r)_n}{n!}p_V^n(1-p_V)^{c/r},\;\;\;n\geq 0,
\end{equation*}
where $(x)_n = x(x+1)\cdots(x+n-1)$ is the Pochhammer symbol. Moreover, the unique stationary distribution of $X$ is the gamma distribution
\begin{equation*}
p(x) = \frac{1}{\Gamma(c/r)}x^{c/r-1}e^{-\lambda x},\;\;\;x>0.
\end{equation*}
\end{corollary}

The following theorem shows that the stationary distribution of the GDDMC model also converges to that of the PDMP model as $V\rightarrow \infty$.

\begin{theorem}
Suppose that $r > L_c/\lambda$. Then $\pi_V\Rightarrow\pi$ as $V\rightarrow \infty$.
\end{theorem}

\begin{proof}
For any $V\geq 1$ and $n\geq0$, let
\begin{equation*}
\phi(x) = x,\;\;\;\psi(x) = -dx+\frac{c(x)}{\lambda}.
\end{equation*}
By \cite[Lemma 4.9.5]{ethier2009markov}, it is easy to check that
\begin{equation*}
\phi(X_V(t))-\int_0^t\psi(X_V(s))ds
\end{equation*}
is a supermartingale whenever $\Enum\phi(X_V(0))<\infty$. This fact, together with \cite[Lemma 4.9.13]{ethier2009markov}, shows that $\{\pi_V\}$ is relatively compact. Since the martingale problems for $\mathcal{A}_V$ and $\mathcal{A}$ are both well posed and since $X_V\Rightarrow X$ as $V\rightarrow \infty$, it follows from \cite[Theorem 4.9.12]{ethier2009markov} that the weak limit of any weakly convergent subsequence of $\{\pi_V\}$ must be a stationary distribution of $X$. Since the stationary distribution of $X$ is unique, all weakly convergent subsequences of $\{\pi_V\}$ must converge weakly to the same limit, which gives the desired result.
\end{proof}

\section{Applications in bursty stochastic gene regulatory networks}
In this section, we propose a mesoscopic GDDMC model of stochastic gene regulatory networks with bursting dynamics and then apply our limit theorem to discuss its limit behavior. Gene regulatory networks can be tremendously complex, involving numerous feedback loops and signaling steps. A schematic diagram of a gene regulatory network is depicted in Fig. \ref{network}(a), where each node represents a gene and each edge represents a feedback relation. A gene regulatory network is usually a directed graph with two types of arrows depicted in Fig. \ref{network}(b), which represent the regulation of an output gene by an input gene via positive or negative feedback. In addition, we also allow a gene to regulate itself via positive or negative autoregulation, as depicted in Fig. \ref{network}(c).
\begin{figure}[!htb]
\centerline{\includegraphics[width=0.8\textwidth]{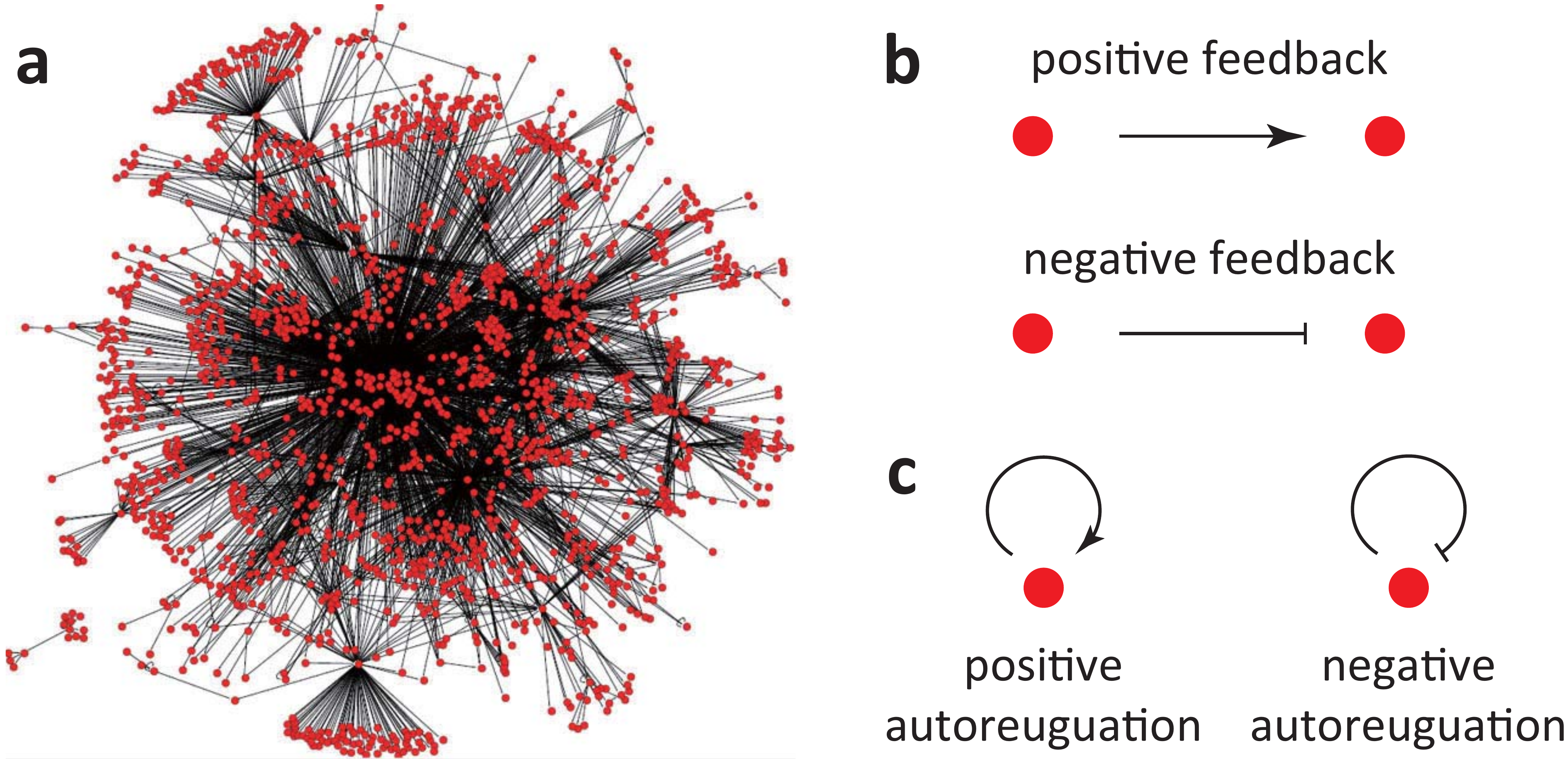}}
\caption{\textbf{Gene regulatory network in living cells.} (a) Schematic diagram of a gene regulatory network, where each red node represents a gene and each black edge represents a feedback relation. (b) Two types of feedback relations of an input gene on an output gene. (c) Positive and negative autoregulation of a gene on itself.}\label{network}
\end{figure}

We then focus on the single-cell gene expression kinetics of a bursty stochastic gene regulatory network. Suppose that the network is composed of $d$ different genes whose gene products are denoted by $P_1,P_2,\cdots,P_d$. For each $1\leq i\leq d$, let $N_i(t)$ denoted the copy number of the protein $P_i$ in an individual cell at time $t$ and let
\begin{equation*}
N(t) = (N_1(t),N_2(t),\cdots,N_d(t))
\end{equation*}
denote the copy number process. Then the concentration process $X_V(t) = N(t)/V$ can be modeled as a $d$-dimensional GDDMC on the lattice
\begin{equation*}
E_V = \set{\tfrac{n}{V}:n=(n_1,n_2,\ldots,n_d)\in\Nnum^d},
\end{equation*}
where $V$ is a scaling parameter. For each $1\leq i\leq d$, let $e_i = (0,\cdots,0,1,0,\cdots,0)$ denote the vector whose $i$th component is 1 and the other components are all zero. Each protein $P_i$ can be synthesized or degraded. The degradation of $P_i$ corresponds to a transition of $X_V$ from $n/V$ to $(n-e_i)/V$ with transition rate
\begin{equation*}
q\left(\tfrac{n}{V},\tfrac{n-e_i}{V}\right) = r_in_i,
\end{equation*}
where $r_i$ is the degradation rate of $P_i$. The synthesis of $P_i$ could occur in random bursts. The synthesis of $P_i$ corresponds to a transition of $X_V$ from $n/V$ to $(n+me_i)/V$ with transition rate
\begin{equation*}
q\left(\tfrac{n}{V},\tfrac{n+me_i}{V}\right) = c_i\left(\tfrac{n}{V}\right)p_i(V,m),\;\;\;m\geq 1,
\end{equation*}
where $c_i(n/V)$ is the effective transcription rate of gene $i$ and $p_i(V,\cdot)$ is the probability distribution of the burst size of $P_i$, as explained in Section 4. The transcription rate of each gene is affected by other genes according to the topology of the gene regulatory network. For each $1\leq i\leq d$, let $E_i$ denote the set of genes that positively regulate gene $i$ and let $I_i$ denote the set of genes that negatively regulate gene $i$. Then the effective transcription rate of gene $i$ is assumed to be governed by the function
\begin{equation*}
c_i(x) = \frac{s_i+\sum_{j\in E_i}x_j^{\mu_{ji}}}
{1+\sum_{j\in E_i}x_j^{\mu_{ji}}+\sum_{j\in I_i}x_j^{\nu_{ji}}},
\end{equation*}
where $s_i$ is a basal transcription rate and the other terms characterize the effects that other genes exert on gene $i$ \cite{rice2004reconstructing}. This influence can be excitatory or inhibitory. The influence of an excitatory gene $j\in E_i$ on gene $i$ is incorporated via the Hill-like coefficient $\mu_{ji}>0$. Similarly, the influence of an inhibitory gene $j\in I_i$ on gene $i$ is incorporated via the Hill-like coefficient $\nu_{ji}>0$. These Hill-like coefficients control the nonlinear dependence of output nodes on input nodes.

Two special burst-size distributions deserve special attention. If
\begin{equation}\label{geometric}
p_i(V,m) = p_i(V)^m(1-p_i(V)),
\end{equation}
then the the burst size of $P_i$ is geometrically distributed, as discussed in Section 4, and we assume that the mean burst size scales with the parameter $V$ as
\begin{equation}\label{scaling}
\frac{p_i(V)}{1-p_i(V)} = \frac{V}{\lambda_i}.
\end{equation}
In recent years, however, there has been evidence showing that the burst size may not be geometrically distributed in eukaryotic cells \cite{elgart2011connecting, schwabe2012transcription, kuwahara2015beyond}. In particular, a molecular ratchet model of gene expression \cite{schwabe2012transcription} predicts a peaked burst-size distribution that resembles the negative binomial distribution
\begin{equation}\label{NB}
p_i(V,m) = \frac{(\alpha_i)_n}{n!}p_i(V)^n(1-p_i(V))^{\alpha_i},
\end{equation}
where $\alpha_i>0$ is a constant. When $\alpha_i = 1$, the negative binomial distribution reduces to the geometric distribution \eqref{geometric}. Burst-size distributions under more complicated biochemical mechanisms can be found in \cite{schwabe2012transcription}. Using the Laplace transform, it is not hard to verify that under the scaling relation \eqref{scaling}, the negative binomial distribution \eqref{NB} converges weakly to the gamma distribution
\begin{equation*}
\mu_i(dx) = \frac{\lambda_i^{\alpha_i}}{\Gamma(\alpha_i)}x^{\alpha_i-1}e^{-\lambda_ix}dx
\end{equation*}
as $V\rightarrow\infty$ and the three conditions listed in \eqref{conditions} are satisfied with the condition (c) being relaxed as discussed in Remark \ref{relax}. If $\alpha_i$ is an integer, then the gamma distribution reduces to an Erlang distribution. This is also consistent with recent studies which used Erlang distributed burst sizes to model molecular memory \cite{qiu2019analytical}.

Under the above framework, the GDDMC model $X_V$ of a bursty stochastic gene regulatory network is associated with the operator
\begin{equation*}
\mathcal{A}_Vf\left(\tfrac{n}{V}\right)
= \sum_{i=1}^dr_in_i\left[f\left(\tfrac{n-e_i}{V}\right)-f\left(\tfrac{n}{V}\right)\right]
+\sum_{i=1}^dc_i\left(\tfrac{n}{V}\right)\sum_{m=1}^{\infty}p_i(V,m)
\left[f\left(\tfrac{n+me_i}{V}\right)-f\left(\tfrac{n}{V}\right)\right].
\end{equation*}
According to our theory, the limit process of $X_V$ is a PDMP $X = \{X(t):t\geq 0\}$ associated with the operator
\begin{equation*}
\mathcal{A}f(x) = -\sum_{i=1}^dr_ix_i\partial_if(x)
+\sum_{i=1}^dc_i(x)\int_0^{\infty}[f(x+ye_i)-f(x)]\mu_i(dy).
\end{equation*}
In particular, if $p_i(V,\cdot)$ is geometrically distributed, then $\mu_i$ is exponentially distributed. If $p_i(V,\cdot)$ is negative binomially distributed, then $\mu_i$ is gamma distributed. In previous works, many authors added independent white noises to the mean field dynamics of a gene regulatory network \cite{rice2004reconstructing}. Compared with these studies, our PDMP model provides a clearer description of the source of stochasticity involved in the network.

The limit behavior of the concentration process $X_V$ is stated rigorously in the following theorem.

\begin{theorem}
Suppose that the three conditions in \eqref{conditions} are satisfied. Let $\nu_{V}$ be the initial distribution of the GDDMC model $X_V$ of a stochastic gene regulatory network and let $\nu$ be the initial distribution of the PDMP model $X$. If $\nu_V\Rightarrow\nu$ as $V\rightarrow\infty$, then $X_V\Rightarrow X$ in $D(\Rnum_+,\Rnum_+^d)$ as $V\rightarrow\infty$.
\end{theorem}

\section{Proof of Theorems \ref{thmwellposed} and \ref{MC}}
In this section, we shall prove that $X$ and $X_V$ are the unique solutions to the martingale problems for $\mathcal{A}$ and $\mathcal{A}_V$, respectively. Before doing these, we introduce some notation. Let $S$ be a metric space and let $S^\Delta$ be the one-point compactification of $S$. Let $\mathcal{R}$ be a linear operator on $C_0(S)$. Then $\mathcal{R}$ can be extended to a linear operator $\mathcal{R}^\Delta$ on $C(S^\Delta)$ with domain
\begin{equation*}
\mathcal{D}(\mathcal{R}^\Delta) = \set{f\in C(S^\Delta):(f-f(\Delta))|_{S}\in\mathcal{D}(\mathcal{R})}
\end{equation*}
defined by
\begin{equation*}
(\mathcal{R}^\Delta f)|_{S} = \mathcal{R}(f-f(\Delta))|_{S},\;\;\;\mathcal{R}^\Delta f(\Delta) = 0.
\end{equation*}

We shall first prove that $X$ is a solution to the martingale problem for $\mathcal{A}$. To this end, we need the following lemmas.

\begin{lemma}\label{SLLN}
For each $n\geq 1$, let $Z_n$ be the $n$th jump vector of $X$ as defined in \eqref{jumpvector}. Then
\begin{equation*}
\lim_{n\rightarrow\infty}Z_1+\cdots+Z_n = \infty,\;\;\;\textrm{a.s.}
\end{equation*}
\end{lemma}

\begin{proof}
By the construction of the jump vectors, it is easy to see that the distribution of each $Z_n$ is a convex combination of $\mu_1,\cdots,\mu_N$. Let $\{X_{mn}:1\leq m\leq N,n\geq 1\}$ be an independent random array such that $X_{mn}$ has the distribution $\mu_m$. Then for each $n\geq 1$, there must exists a random variable $T_n$ with values in $\{1,2,\cdots,N\}$ such that $Z_n$ and $X_{T_n,n}$ has the same distribution. Since $1\leq T_n\leq N$, it follows from the strong law of large numbers that
\begin{equation*}
\lim_{n\rightarrow\infty}X_{T_1,1}+\cdots+X_{T_n,n} = \infty,\;\;\;\textrm{a.s.}
\end{equation*}
This gives the desired result.
\end{proof}

\begin{lemma}\label{lemclosure}
$X$ is a solution to the martingale problem for $\mathcal{A}^\Delta$.
\end{lemma}

\begin{proof}
For any $f\in\mathcal{D}(\mathcal{A})$, it is easy to check that $\mathcal{A}f\in C_c(\Rnum_+^d)$. This shows that $\mathcal{A}$ is a linear operator on $C_0(\Rnum_+^d)$ and thus $\mathcal{A}^\Delta$ is a well defined linear operator on $C((\Rnum_+^d)^{\Delta})$. Without loss of generality, we assume that $X_0 = x\in\Rnum_+^d$. Let $\phi(t,x)$ be the global flow defined in \eqref{flow}. For any $f\in\mathcal{D}(\mathcal{A})$, we have
\begin{equation*}
\mathcal{A}f(\phi(t,x)) = \frac{d}{dt}f(\phi(t,x))
+\sum_{i=1}^Nc_i(\phi(t,x))\int_{\Rnum_+^d}[f(\phi(t,x)+y)-f(\phi(t,x))]\mu_i(dy).
\end{equation*}
For any $g\in C^1[0,\infty)$ with $g(0) = 0$ and $t\geq 0$, it is easy to check that
\begin{equation*}
\Enum g(t\wedge T_1) = \int_0^t g'(s)\Pnum(T_1>s)ds.
\end{equation*}
For each $m\geq 1$, let $T_m$ be the $m$th jump time of $X$ and let $Z_m$ be the $m$th jump vector of $X$. Applying the above two equations gives rise to
\begin{align*}
&\; \Enum f(\phi(t\wedge T_1,x))-f(x)-\Enum\int_0^{t\wedge T_1}\mathcal{A}f(\phi(s,x))ds \\
=&\; -\sum_{i=1}^N\Enum\int_0^{t\wedge T_1}c_i(\phi(s,x))
\int_{\Rnum_+^d}[f(\phi(s,x)+y)-f(\phi(s,x))]\mu_i(dy) \\
=&\; -\sum_{i=1}^N\int_0^tc_i(\phi(s,x))e^{-\int_0^sc(\phi(u,x))du}
\int_{\Rnum_+^d}[f(\phi(s,x)+y)-f(\phi(s,x))]\mu_i(dy).
\end{align*}
Since the trajectory of $X$ coincides with that of $\phi(t,x)$ before $T_1$, we have
\begin{align*}
&\; \Enum f(X(t\wedge T_1))-\Enum f(\phi(t\wedge T_1,x)) \\
=&\; \Enum[f(\phi(T_1,x)+Z_1)-f(\phi(T_1,x))]I_{\set{T_1\leq t}}\\
=&\; \int_0^t\Pnum(T_1\in ds)\Enum[f(\phi(s,x)+Z_1)-f(\phi(s,x))] \\
=&\; \sum_{i=1}^N\int_0^tc_i(\phi(s,x))e^{-\int_0^sc(\phi(u,x))du}
\int_{\Rnum_+^d}[f(\phi(s,x)+y)-f(\phi(s,x))]\mu_i(dy).
\end{align*}
Adding the above two equations gives rise to
\begin{equation*}
\Enum f(X(t\wedge T_1))-f(x) = \Enum\int_0^{t\wedge T_1}\mathcal{A}f(X(s))ds.
\end{equation*}
By induction and the construction of the PDMP limit, it is not difficult to prove that
\begin{equation}\label{finite}
\Enum f(X(t\wedge T_m))-f(x) = \Enum\int_0^{t\wedge T_m}\mathcal{A}f(X(s))ds,\;\;\;
\textrm{for any}\; m\geq 1.
\end{equation}
To proceed, we select a sequence $\{f_n:n\geq 1\}\subset\mathcal{D}(\mathcal{A}^{\Delta})$ such that $f_n\leq 0$, $f_n(\Delta) = 0$, and $\{f_n\}$ separates points in $(\Rnum_+^d)^{\Delta}$, which means that for any $x,y\in(\Rnum_+^d)^{\Delta}$ and $x\neq y$, there exists $n\geq 1$ such that $f_n(x)\neq f_n(y)$. Taking $m\rightarrow \infty$ in \eqref{finite} and applying Fatou's lemma, we obtain that
\begin{align*}
\Enum f_n(X(t))-f(x) &\geq \Enum\limsup_{m\rightarrow\infty}f_n(X(t\wedge T_m))-f(x) \\
&\geq \Enum\int_0^{t\wedge T_\infty}\mathcal{A}^\Delta f_n(X(s))ds
= \Enum\int_0^t\mathcal{A}^\Delta f_n(X(s))ds.
\end{align*}
This fact, together with the Markov property of $X$, shows that
\begin{equation}\label{supermtgl}
f_n(X(t))-f_n(X(0))-\int_0^t\mathcal{A}^\Delta f_n(X(s))ds
\end{equation}
is a submartingale for each $n$. Doob's regularity theorem \cite[Theorem 65.1]{rogers2000diffusions} claims that a right-continuous submartingale must be c\`{a}dl\`{a}g almost surely. Thus, the process $f_n(X)$ must have left limits for each $n$. Since $\{f_n\}$ separates points in $(\Rnum_+^d)^{\Delta}$, the process $X$ must also have left limits. We next claim that for any $t\geq 0$,
\begin{equation}\label{infinity}
\lim_{m\rightarrow\infty}X(t\wedge T_m) = X(t),\;\;\;\textrm{a.s.}
\end{equation}
This equality is obvious when $t<T_\infty$. We next consider the case of $T_\infty\leq t$. In this case, we only need to prove that
\begin{equation*}
\lim_{m\rightarrow\infty}X(T_m) = \Delta,\;\;\;\textrm{a.s.}
\end{equation*}
If this is false, then there is a positive probability such that $\{X(T_m)\}$ is a bounded sequence. Suppose that $|X(T_m)|\leq M$ for any $m\geq 1$. It is worth noting that
\begin{equation*}
X(T_m) = X(0)+\sum_{k=1}^m[X(T_k-)-X(T_{k-1})]+\sum_{k=1}^mZ_k,
\end{equation*}
where $T_0 = 0$. Since $F$ is Lipschitz, we have
\begin{equation*}
|\phi(t,x)-x| \leq \int_0^t|F(\phi(s,x))|ds \leq |F(x)|t+L_F\int_0^t|\phi(s,x)-x|ds.
\end{equation*}
By Gronwall's inequality, we have
\begin{align*}
\sum_{k=1}^m|X(T_k-)-X(T_{k-1})| &= \sum_{k=1}^m|\phi(T_k-T_{k-1},X(T_{k-1}))-X(T_{k-1})|\\
&\leq \sum_{k=1}^m\sup_{0\leq x\leq M}|F(x)|e^{L_FT_\infty}(T_k-T_{k-1}) \\
&\leq \sup_{0\leq x\leq M}|F(x)|T_\infty e^{L_FT_\infty} < \infty.
\end{align*}
This fact, together with Lemma \ref{SLLN}, shows that $X_{T_m}\rightarrow\Delta$, which leads to a contradiction. Thus we have proved \eqref{infinity}. Taking $m\rightarrow\infty$ in \eqref{finite} and applying the dominated convergence theorem, we obtain that for any $f\in\mathcal{D}(\mathcal{A}^\Delta)$,
\begin{equation*}
\Enum f(X(t))-f(x) = \Enum\int_0^t\mathcal{A}^\Delta f(X(s))ds.
\end{equation*}
This fact, together with the Markov property of $X$, shows that
\begin{equation}\label{supermtgl}
f(X(t))-f(X(0))-\int_0^t\mathcal{A}^\Delta f(X(s))ds
\end{equation}
is indeed a martingale, which gives the desired result.
\end{proof}

To proceed, we recall the following important concept \cite[Section 3.4]{ethier2009markov}.

\begin{definition}
Let $S$ be a metric space and let $\{f_n\}$ be a sequence in $B(S)$. We say that $\{f_n\}$ \emph{converges boundedly and pointwise} or \emph{bp-converges} to $f\in B(S)$ if $\{f_n\}$ is uniformly bounded and $f_n(x)\rightarrow f(x)$ for each $x\in S$. A set $M\subset B(S)$ is called \emph{bp-closed} if whenever $\{f_n\}\subset M$ and $\{f_n\}$ bp-converges to $f$, we have $f\in M$. The \emph{bp-closure} of $M$ is defined as the smallest bp-closed subset of $B(S)$ that contains $M$.
\end{definition}

We still need the following lemma, whose proof can be found in \cite[Theorem 4.3.8]{ethier2009markov}.

\begin{lemma}\label{restriction}
Let $S$ be a metric space and let $U$ be an open subset of $S$. Let $\mathcal{R}$ be an operator on $B(S)$ with domain $\mathcal{D}(\mathcal{R})\subset C_b(S)$ and graph $\mathcal{G}(\mathcal{R})$. Suppose that $Y$ is a solution to the martingale problem for $\mathcal{R}$. If $\Pnum(Y_0\in U) = 1$ and $(I_U,0)$ is in the bp-closure of $\mathcal{G}(\mathcal{R})$, then $\Pnum(Y\in D(\Rnum_+,U)) = 1$.
\end{lemma}

The following lemma plays an important role in proving the nonexplosiveness of $X$.

\begin{lemma}\label{closure}
$(I_{\Rnum_+^d},0)$ is in the bp-closure of $\mathcal{G}(\mathcal{A}^\Delta)$.
\end{lemma}

\begin{proof}
For each $n\geq 1$, there exists $g_n\in C_c^1(\Rnum_+)$ satisfying $0\leq g_n\leq 1$ and
\begin{equation}\label{uryson}
\begin{cases}
\;g_n(x) = 1, &0\leq x\leq n,\\
\;g_n(x) = 0, &x\geq 3n,\\
\;|g_n'(x)|<\frac{1}{n}, &x\geq 0.
\end{cases}
\end{equation}
Let $f_n$ be a function on $\Rnum_+^d$ defined by $f_n(x) = g_n(|x|)$. Then $f_n\in C_c^1(\Rnum_+^d)$ and $|\nabla f_n(x)| < 1/n$. Moreover, it is easy to check that $\mathcal{A}f_n(x) = 0$ for any $|x|\geq 3n$. Since $F$ and $c_i$ are Lipschitz functions, for any $x\in \Rnum_+^d$,
\begin{equation*}
|F(x)|\leq |F(0)|+L_F|x|,\;\;\;c_i(x)\leq c_i(0)+L_{c_i}|x|.
\end{equation*}
For any $|x|<3n$, it follows from the mean value theorem that
\begin{equation}\label{eqlem13}
\begin{split}
|\mathcal{A}f_n(x)| &\leq \frac{|F(x)|}{n}+\sum_{i=1}^N\frac{c_i(x)}{n}\int_{\Rnum_+^d}|y|\mu_i(dy)\\
&\leq |F(0)|+3L_F+\sum_{i=1}^N[c_i(0)+3L_{c_i}]\int_{\Rnum_+^d}|y|\mu_i(dy),
\end{split}
\end{equation}
which shows that $\{\mathcal{A}f_n\}$ is uniformly bounded. For any $x\in \Rnum_+^d$, whenever $n\geq |x|$, we have
\begin{equation*}
|\mathcal{A}f_n(x)| \leq \sum_{i=1}^Nc_i(x)\int_{|y|>n-|x|}|f_n(x+y)-f_n(x)|\mu_i(dy)
\leq \sum_{i=1}^Nc_i(x)\mu_i(\{y:|y|>n-|x|\}),
\end{equation*}
which tends to zero as $n\rightarrow\infty$. Thus, $(f_n,\mathcal{A}f_n)$ bp-converges to $(I_{\Rnum_+^d},0)$.
\end{proof}

\begin{lemma}\label{proexistence}
$X$ is a solution to the martingale problem for $\mathcal{A}$.
\end{lemma}

\begin{proof}
If we take
\begin{equation*}
S = (\Rnum_+^d)^\Delta,\;\;\;U = \Rnum_+^d,\;\;\;\mathcal{R} = \mathcal{A}^\Delta,
\end{equation*}
then it follows from Lemmas \ref{lemclosure} and \ref{closure} that all the conditions in Lemma \ref{restriction} are satisfied. Then $X$ has sample paths in $D(\Rnum_+,\Rnum_+^d)$. By the definition of $\mathcal{A}^\Delta$, it is easy to check that $X$ is also a solution to the martingale problem for $\mathcal{A}$.
\end{proof}

We still need to prove the uniqueness of the martingale problem for $\mathcal{A}$. To this end, we define a sequence of auxiliary operators $\set{\mathcal{A}_n}$ with bounded coefficients. For each $n\geq 1$, let $\mathcal{A}_n$ be a L\'{e}vy-type operator on $B(\Rnum_+^d)$ with domain $D(\mathcal{A}_n) = C_c^1(\Rnum_+^d)$ defined as
\begin{equation*}
\mathcal{A}_nf(x) = \sum_{i=1}^dF^{(n)}_i(x)\partial_if(x)
+\sum_{i=1}^Nc_i^{(n)}(x)\int_{\Rnum_+^d} [f(x+y)-f(x)]\mu_i(dy),
\end{equation*}
where
\begin{equation*}
F^{(n)}(x) = F\left(\tfrac{|x|\wedge n}{|x|}\cdot x\right),\;\;\;
c_i^{(n)}(x) = c_i\left(\tfrac{|x|\wedge n}{|x|}\cdot x\right).
\end{equation*}
For any $f\in C_c^1(\Rnum_+^d)$, it is easy to see that $\mathcal{A}f(x) = \mathcal{A}_nf(x)$ for any $|x|\leq n$. It is convenient to rewrite the operator $\mathcal{A}_n$ as
\begin{equation*}
\mathcal{A}_nf(x) = \sum_{i=1}^db^{(n)}_i(x)\partial_if(x)
+\int_{\Rnum_+^d}[f(x+y)-f(x)-\sum_{i=1}^dy_i\partial_if(x)I_{\{|y|<1\}}]\eta(x,dy),
\end{equation*}
where
\begin{equation*}
b^{(n)}(x) = F^{(n)}(x)+\sum_{i=1}^Nc_i^{(n)}(x)\int_{\{|y|<1\}}y\mu_i(dy),\;\;\;
\eta(x,dy) = \sum_{i=1}^Nc_i^{(n)}(x)\mu_i(dy).
\end{equation*}

\begin{lemma}\label{lemBn}
For each $\nu\in\mathcal{P}(\Rnum_+^d)$, the martingale problem for $(\mathcal{A}_n,\nu)$ is well posed.
\end{lemma}

\begin{proof}
Suppose that there exist $\lambda:\Rnum^d\times S\rightarrow[0,1]$, $\gamma:S\rightarrow\Rnum^d$, and a $\sigma$-finite measure $\nu$ on a measurable space $(S,\mathcal{S})$ such that
\begin{equation*}
\eta(x,\Gamma) = \int_S \lambda(x,u)I_{\Gamma}(\gamma(u))\nu(du),\;\;\;
\textrm{for any}\;\Gamma\in \mathcal{B}(\Rnum_+^d),x\in\Rnum_+^d.
\end{equation*}
In addition, set
\begin{equation*}
S_1 = \set{u\in S:|\gamma(u)|<1},\;\;\;S_2 = \set{u\in S:|\gamma(u)|\geq 1}.
\end{equation*}
By a classical result of Kurtz about the well-posedness of the martingale problem for a L\'{e}vy-type operator \cite[Theorems 2.3 and 3.1]{kurtz2011equivalence}, the martingale problem for $(\mathcal{A}_n,\nu)$ is well posed if there exists a constant $M>0$ such that for any $x,y\in\Rnum_+^d$, the following three conditions are satisfied:
\begin{gather}
|b^{(n)}(x)|+\int_{S_1}\lambda(x,u)|\gamma(u)|^2\nu(du)+\int_{S_2}\lambda(x,u)|\gamma(u)|\nu(du)<M,\nonumber\\
|b^{(n)}(x)-b^{(n)}(y)|\leq M|x-y|,\label{threeconditions}\\
\int_{S}|\lambda(x,u)-\lambda(y,u)|\cdot|\gamma(u)|\nu(du)\leq M|x-y|.\nonumber
\end{gather}
To verify the above three conditions, let $S = \Rnum_+^d\times\{1,2,\cdots,N\}$ and for each $(u,i)\in S$, choose
\begin{equation*}
\lambda(x,u,i) = \frac{c_i^{(n)}(x)}{\beta_n},\;\;\;\gamma(u,i) = u,\;\;\;
\nu(du,di) = \beta_n\mu_i(du)n(di),
\end{equation*}
where $n(di)$ is the counting measure on $\{1,2,\cdots,N\}$ and
\begin{equation*}
\beta_n = \left\|\sum_{i=1}^Nc_i^{(n)}\right\|+1.
\end{equation*}
Then for any Borel set $\Gamma\subset\Rnum_+^d$,
\begin{equation*}
\int_S\lambda(x,u,i)I_{\Gamma}(\gamma(u,i))\nu(du,di)
= \sum_{i=1}^Nc_i^{(n)}(x)\int_{\Rnum_+^d}I_{\Gamma}(u)\mu_i(du)\\
= \eta(x,\Gamma).
\end{equation*}
We next check the three conditions listed in \eqref{threeconditions}. For any $x,y\in \Rnum^d$, it is easy to see that
\begin{equation*}
\int_{S_1}\lambda(x,u,i)|\gamma(u,i)|^2\nu(du,di)+\int_{S_2}\lambda(x,u,i)|\gamma(u,i)|\nu(du,di)
\leq \beta_n\int_{\Rnum_+^d}|x|\mu_i(dx).
\end{equation*}
Moreover, we have
\begin{align*}
\int_{S}|\lambda(x,u,i)-\lambda(y,u,i)|\cdot|\gamma(u,i)|\nu(du,di)
&\leq \sum_{i=1}^N|c_i^{(n)}(x)-c_i^{(n)}(y)|\int_{\Rnum_+^d}|x|\mu_i(dx)\\
&\leq \sum_{i=1}^NL_{c_i}\int_{\Rnum_+^d}|x|\mu_i(dx)|x-y|.
\end{align*}
Since both $b^{(n)}$ and $c_i^{(n)}$ are bounded and Lipschitz, we obtain the desired result.
\end{proof}

To proceed, we recall the following definition \cite[Section 4.6]{ethier2009markov}.

\begin{definition}
The notation is the same as in Definition \ref{martingale}. Let $U$ be an open subset of $S$ and let \begin{equation*}
\tau=\inf\{t\geq 0: Y(t)\not\in U\textrm{\;or\;}Y(t-)\not\in U\}
\end{equation*}
be the first exit time of $Y$ from $U$. For any $\nu\in\mathcal{P}(S)$, we say that $Y$ is a solution to the \emph{stopped martingale problem} for $(\mathcal{R},\nu, U)$ if \\
(a) $Y$ has the initial distribution $\nu$, \\
(b) $Y(\cdot) = Y(\cdot\wedge\tau)$ almost surely, and\\
(c) for any $f\in\mathcal{D}(\mathcal{R})$,
\begin{equation*}
f(Y(t))-f(Y(0))-\int_0^{t\wedge\tau}\mathcal{R}f(Y(s))ds
\end{equation*}
is a martingale with respect to the natural filtration generated by $Y$.
\end{definition}

We are now in a position to prove Theorem \ref{thmwellposed}.

\begin{proof}[Proof of Theorem \ref{thmwellposed}]
By Lemma \ref{proexistence}, $X$ is a solution to the martingale problem for $\mathcal{A}$. We next prove the uniqueness of the martingale problem. For each $n\geq 1$, let $U_n = \{x\in\Rnum_+^d: |x|<n\}$. It is obvious that $\mathcal{A}_nf|_{U_n} = (\mathcal{A}f)|_{U_n}$ for any $f\in D(\mathcal{A})$. By Lemma \ref{lemBn} and \cite[Theorem 4.6.1]{ethier2009markov}, there exists a unique solution to the stopped martingale problem for $(\mathcal{A},\nu,U_n)$. Since $\Rnum_+^d$ is the union of all $U_n$, it follows from \cite[Theorem 4.6.2]{ethier2009markov} that the martingale problem for $\mathcal{A}$ is unique.
\end{proof}

We next prove Theorem \ref{MC}.

\begin{proof}[Proof of Theorem \ref{MC}]
In analogy to the proof of Lemma \ref{lemclosure}, we can prove that $X_V$ is a solution to the martingale problem for $\mathcal{A}_V^{\Delta}$. Let $f$ be a function on $E_V$ defined by
\begin{equation*}
f\left(\tfrac{n}{V}\right)=\sum_{i=1}^{d}\tfrac{n_i}{V}+1.
\end{equation*}
Direct computations show that
\begin{align*}
\mathcal{A}_Vf\left(\tfrac{n}{V}\right)
=&\; \sum_{i=1}^dF_i\left(\tfrac{n}{V}\right)
+\tfrac{1}{V}\sum_{i=1}^Nc_i\left(\tfrac{n}{V}\right)\sum_{j=1}^d\sum_{m\in \Nnum^d}m_jp_i(V,m)\\
\leq&\; d|F\left(\tfrac{n}{V}\right)|
+\tfrac{d}{V}\sum_{i=1}^Nc_i\left(\tfrac{|n|}{V}\right)\sum_{m\in\Nnum^d}|m|p_i(V,m)\\
\leq&\; d|F(0)|+L_F\tfrac{n}{V}
+\tfrac{d}{V}\sum_{i=1}^N\left[c_i(0)+L_{c_i}\tfrac{|n|}{V}\right]\sum_{m\in \Nnum^d}|m|p_i(V,m)\\
\leq&\; \bigg[d|F(0)|+L_F+\tfrac{d}{V}\sum_{i=1}^N\left[c_i(0)+L_{c_i}\right]\sum_{m\in \Nnum^d}|m|p_i(V,m)\bigg]f\left(\tfrac{n}{V}\right).
\end{align*}
By \cite[Theorem 2.25]{chen2004markov}, $X_V$ is nonexplosive and thus is a solution to the martingale problem for $\mathcal{A}_V$. By using the localization technique as in the proof of Lemma \ref{lemBn}, it is easy to prove that $X_V$ is the unique solution to the martingale problem for $(\mathcal{A}_V,\nu_V)$.
\end{proof}

\section{Proof of Theorem \ref{thmergodicity}}\label{coupling}
In this section, we shall prove the exponential ergodicity of $X$. For simplicity of notation, we only consider the case of $N = 1$, where the operator $\mathcal{A}$ has the form of \eqref{eqoperator}. The proof of the general case is totally the same.

To prove the exponential ergodicity of $X$, we construct a coupling operator as follows. Let $\tilde{\mathcal{A}}$ be an operator on $B(\Rnum_{+}^d\times\Rnum_{+}^d)$ with domain $D(\tilde{\mathcal{A}}) = C_c^1(\Rnum_{+}^d\times\Rnum_{+}^d)$ defined by
\begin{align*}
\tilde{\mathcal{A}}f(x,y) =&\; \langle F(x), \nabla_xf(x,y)\rangle+\langle F(y),\nabla_yf(x,y)\rangle\\
&\; +(c(x)\wedge c(y))\int_{\Rnum_+^d}[f(x+z,y+z)-f(x,y)]\mu(dz)\\
&\; +(c(x)-c(y))^{+}\int_{\Rnum_+^d}[f(x+z,y)-f(x,y)]\mu(dz)\\
&\; +(c(x)-c(y))^{-}\int_{\Rnum_+^d}[f(x,y+z)-f(x,y)]\mu(dz).
\end{align*}

The following lemma, whose proof can be found in \cite[Theorem 4.5.4]{ethier2009markov}, plays an important role in proving the existence of the martingale problem.

\begin{lemma}\label{exsistence}
Let $S$ be a locally compact separable metric space and let $\mathcal{R}$ be a densely defined linear operator on $C_0(S)$ with domain $\mathcal{D}(\mathcal{R})$. Suppose that $\mathcal{R}$ satisfies the \emph{positive maximum principle}, that is, if $f\in \mathcal{D}(\mathcal{R})$ attains its maximum at $x_0\in S$, then $\mathcal{R}f(x_0)\leq 0$. Then for any $\nu\in \mathcal{P}(S^\Delta)$, there exists a solution to the martingale problem for $(\mathcal{R}^\Delta,\nu)$ with sample paths in $D(\Rnum_+,S^\Delta)$.
\end{lemma}

The following lemma gives the existence of the martingale problem for $\tilde{\mathcal{A}}$.

\begin{lemma}\label{closure2}
For each $\nu\in \mathcal{P}(\Rnum_+^d\times\Rnum_+^d)$, there exists a solution to the martingale problem for $(\tilde{\mathcal{A}},\nu)$.
\end{lemma}

\begin{proof}
It is easy to check that $\tilde{\mathcal{A}}$ is a densely defined linear operator on $C_0(\Rnum_+^d\times\Rnum_+^d)$ and $\tilde{\mathcal{A}}$ satisfies the positive maximum principle. Then by Lemma \ref{exsistence}, there exists a solution $Y$ to the martingale problem for $(\tilde{\mathcal{A}}^\Delta,\nu)$. We next prove that $(I_{\Rnum_+^d\times\Rnum_+^d},0)$ is in the bp-closure of $\mathcal{G}(\tilde{\mathcal{A}}^\Delta)$. To do this, for each $n\geq 1$, we define a function $f_n\in C_c^1(\Rnum_+^d\times\Rnum_+^d)$ by
\begin{equation*}
f_n(x,y) = g_n(|x|)g_n(|y|),
\end{equation*}
where $g_n\in C_c^1(\Rnum_+)$ is the function defined in \eqref{uryson}. It is easy to see that $\tilde{\mathcal{A}}f_n(x,y) = 0$ for any $|x|\geq 3n$ or $|y|\geq 3n$. Moreover, it follows from the mean value theorem that for any $|x|<3n$ and $|y|<3n$,
\begin{align*}
|\tilde{\mathcal{A}}f_n(x,y)|
&\leq \frac{|F(x)|+|F(y)|}{n}+\frac{2c(x)\wedge c(y)}{n}\int_{\Rnum_+^d}|z|\mu(dz)
+\frac{|c(x)-c(y)|}{n}\int_{\Rnum_+^d}|z|\mu(dz)\\
&\leq \frac{|F(x)|+|F(y)|}{n}+\frac{2[c(x)+c(y)]}{n}\int_{\Rnum_+^d}|z|\mu(dz)\\
&\leq 2|F(0)|+6L_F+[2c(0)+6L_c]\int_{\Rnum_+^d}|z|\mu(dz),
\end{align*}
which shows that $\{\tilde{\mathcal{A}}f_n\}$ is uniformly bounded. For any $x,y\in\Rnum_+^d$, whenever $n\geq |x|\vee |y|$, we have
\begin{equation*}
\begin{split}
|\tilde{\mathcal{A}}f_k(x,y)| =&\; (c(x)\wedge c(y))\int_{|z|>n-|x|\vee|y|}|f(x+z,y+z)-f(x,y)|\mu(dz)\\
&\; +(c(x)-c(y))^{+}\int_{|z|>n-|x|}|f(x+z,y)-f(x,y)|\mu(dz)\\
&\; +(c(x)-c(y))^{-}\int_{|z|>n-|y|}|f(x,y+z)-f(x,y)|\mu(dz)\\
\leq&\; (c(x)\wedge c(y))\mu(\{z:|z|>n-|x|\vee|y|\})\\
&\; +(c(x)-c(y))^{+}\mu(\{z:|z|>n-|x|\})\\
&\; +(c(x)-c(y))^{-}\mu(\{z:|z|>n-|y|\}),
\end{split}
\end{equation*}
which tends to zero as $n\rightarrow\infty$. Thus $(f_n,\tilde{\mathcal{A}}f_n)$ bp-converges to $(I_{\Rnum_+^d\times\Rnum_+^d},0)$. If we take
\begin{equation*}
S =(\Rnum_+^d\times\Rnum_+^d)^\Delta,\;\;\;U = \Rnum_+^d\times\Rnum_+^d,\;\;\;
\mathcal{R} = \tilde{\mathcal{A}},
\end{equation*}
then all the conditions in Lemma \ref{restriction} are satisfied. Thus, $Y$ has sample paths in $D(\Rnum_+,\Rnum_+^d\times\Rnum_+^d)$. By the definition of $\tilde{\mathcal{A}}^\Delta$, it is easy to see that $Y$ is also a solution to the martingale problem for $(\tilde{\mathcal{A}},\nu)$.
\end{proof}

The following lemma shows that $\tilde{\mathcal{A}}$ indeed is the coupling operator of $\mathcal{A}$.

\begin{lemma}
For any $x,y\in \Rnum_+^d$, let $\delta_{x,y}$ be the point mass at $(x,y)$ and let $(X,Y)$ be a solution to the martingale problem for $(\tilde{\mathcal{A}},\delta_{x,y})$. Then $X$ is solution to the martingale problem for $(\mathcal{A},\delta_x)$ and $Y$ is the solution to the martingale problem for $(\mathcal{A},\delta_y)$.
\end{lemma}

\begin{proof}
For any $n\geq 1$ and $f\in C_c^1(\Rnum_+^d)$, let $h_n$ be a function on $\Rnum_+^d\times\Rnum_+^d$ defined by
\begin{equation*}
h_n(x,y) = f(x)g_n(|y|),
\end{equation*}
where $g_n\in C_c^1(\Rnum_+)$ is the function defined in \eqref{uryson}. It is obvious that $h_n\in C_c^1(\Rnum_+^d\times\Rnum_+^d)$. Therefore,
\begin{equation*}
h_n(X_t,Y_t)-h_n(X_0,Y_0)-\int_0^t \tilde{\mathcal{A}}h_n(X_s,Y_s)ds
\end{equation*}
is a martingale. Since $f$ has a compact support, there exists $\gamma>0$ such that $f(x) = 0$ for all $|x|\geq\gamma$. For any $|x|>\gamma$ or $|y|>3n$, it is easy to see that $\tilde{\mathcal{A}}h_n(x,y) = 0$. Moreover, straightforward calculations show that
\begin{align*}
|\tilde{\mathcal{A}}h_n(x,y)|
\leq&\; \|\langle F,\nabla f\rangle\|+\|f\||g_n'(|y|)||F(y)| \\
&\; +(c(x)\wedge c(y))\int_{\Rnum_+^d}|f(x+z)-f(x)|g_n(|y+z|)\mu(dz)\\
&\; +(c(x)-c(y))^+\int_{\Rnum_+^d}|f(x+z)-f(x)|g_n(|y|)\mu(dz)\\
&\; +c(y)\int_{\Rnum_+^d}|f(x)||g_n(|y+z|)-g_n(|y|)|\mu(dz)
\end{align*}
For any $|x|\leq\gamma$ and $|y|\leq 3n$, applying the mean value theorem yields
\begin{align*}
|\tilde{\mathcal{A}}h_n(x,y)| &\leq \|\langle F,\nabla f\rangle\|
+\frac{1}{n}\|f\||F(y)|+2\|f\|c(x)+\frac{1}{n}\|f\|c(y)\int_{\Rnum_+^d}|z|\mu(dz).
\end{align*}
Since $F$ and $c$ are Lipschitz functions, we have
\begin{equation*}
\frac{1}{n}|F(y)|\leq |F(0)|+3L_F,\;\;\;c(x)\leq c(0)+L_c\gamma,\;\;\;\frac{1}{n}c(y)\leq c(0)+3L_c,
\end{equation*}
which implies that $\{\tilde{\mathcal{A}}h_n\}$ is uniformly bounded. Moreover, it is easy to check that
\begin{equation*}
\lim_{n\rightarrow\infty}h_n(x,y) = f(x),\;\;\;
\lim_{n\rightarrow\infty}\tilde{\mathcal{A}}h_n(x,y) = \mathcal{A}f(x),\;\;\;
\textrm{for any}\;x,y\in\Rnum_+^d.
\end{equation*}
By the dominated convergence theorem,
\begin{equation*}
f(X_t)-f(X_0)-\int_0^t\mathcal{A}f(X_s)ds
\end{equation*}
is also a martingale. Therefore, $X$ is a solution to the martingale problem for $(\mathcal{A},\delta_{x})$. Similarly, $Y$ is a solution to the martingale problem for $(\mathcal{A},\delta_{y})$.
\end{proof}

Let $\{P_t\}$ be the transition semigroup generated by $X$. For any $\nu\in\mathcal{P}(\Rnum_+^d)$, let $\nu P_t$ be the probability measure defined by $\nu P_t(\cdot) = \Pnum_\nu(X_t\in\cdot)$. The following lemma plays an important role in studying the exponential ergodicity of $X$.

\begin{lemma}\label{lemexpon}
Under the conditions in Theorem \ref{thmergodicity}, we have
\begin{equation*}
W(\delta_xP_t,\delta_yP_t)\leq e^{-\tilde{r}t}|x-y|,\;\;\;\textrm{for any}\;x,y\in \Rnum_+^d.
\end{equation*}
\end{lemma}

\begin{proof}
For any $n\geq 1$, let $g_n:\Rnum_+\rightarrow\Rnum_+$ be a function defined by
\begin{equation*}
g_n(t) = \begin{cases}
\frac{n}{2}t^2+\frac{1}{2n}, &0\leq t\leq \frac{1}{n}\\
t, & t>\frac{1}{n},
\end{cases}
\end{equation*}
We then choose $\chi_n\in C_c^1(\Rnum_+^d)$ such that $0\leq \chi_n\leq 1$ and $\chi_n(x) = 1$ for any $0\leq |x|\leq n$. Moreover, let $f_n$ be a function on $\Rnum_+^d\times\Rnum_+^d$ defined as
\begin{equation*}
f_n(x,y) = g_n(|x-y|)\chi_n(x)\chi_n(y).
\end{equation*}
It is easy to check that $f_n\in C^1(\Rnum_+^d\times\Rnum_+^d)$. If $|x-y|>1/n$ and $|x|,|y|\leq n$, we have
\begin{equation*}
\nabla_xf_n(x,y) = -\nabla_yf_n(x,y) = \frac{x-y}{|x-y|}.
\end{equation*}
For any $x,y,z\in\Rnum_+^d$, it is easy to see that $f(x+z,y+z)\leq f(x,y)$. These facts, together with the mean value theorem, show that
\begin{align*}
\tilde{\mathcal{A}}f_n(x,y)
\leq&\; \frac{1}{|x-y|}\langle F(x)-F(y),x-y\rangle \\
&\; +(c(x)-c(y))^+\int_{\Rnum_+^d}|g_n(|x+z-y|)-g_n(|x-y|)|\mu(dz) \\
&\; +(c(x)-c(y))^-\int_{\Rnum_+^d}|g_n(|x-y-z|)-g_n(|x-y|)|\mu(dz) \\
\leq&\; \frac{1}{|x-y|}\langle F(x)-F(y),x-y\rangle+|c(x)-c(y)|\int_{\Rnum_+^d}|z|\mu(dz).
\end{align*}
Since $\langle F(x)-F(y),x-y\rangle\leq -r|x-y|^2$, for any $|x-y|>1/n$ and $|x|,|y|\leq n$,
\begin{equation}\label{eqafn}
\tilde{\mathcal{A}}f_n(x,y) \leq -r|x-y|+L_c\int_{\Rnum_+^d}|z|\mu(dz)|x-y|
= \tilde{r}|x-y| = \tilde{r}f_n(x,y).
\end{equation}
Let $(X,Y)$ be a solution to the martingale problem for $\mathcal{A}$. Since $f_n(X_t,Y_t)-\int_0^t \tilde{\mathcal{A}}f_n(X_s,Y_s)ds$ is a martingale, it follows from \cite[Lemma 4.3.2]{ethier2009markov} that
\begin{equation*}
e^{\tilde{r}t}f_n(X_t,Y_t)
-\int_0^t e^{\tilde{r}s}[\tilde{r}f_n(X_s,Y_s)+\tilde{\mathcal{A}}f_n(X_s,Y_s)]ds
\end{equation*}
is also a martingale. Let $T_n$ be a stopping time defined by
\begin{equation*}
T_n = \inf\{t>0: |X_t-Y_t|<1/n\;\mbox{or}\;|X_t|>n\;\mbox{or}\;|Y_t|>n\}.
\end{equation*}
For any $x,y\in\Rnum_+^d$ and $x\neq y$, it is obvious that $|x-y|> 1/n$ and $|x|,|y|\leq n$ when $n$ is sufficiently large. For any $m\geq n$, it follows from \eqref{eqafn} that
\begin{equation}\label{keyeq}
\begin{split}
&\; \Enum_{(x,y)}e^{\tilde{r}(t\wedge T_n)}f_m(X_{t\wedge T_n},Y_{t\wedge T_n})\\
=&\; f_m(x,y)+\Enum_{(x,y)}
\int_0^{t\wedge T_n}e^{\tilde{r}s}[\tilde{r}f_m(X_s,Y_s)+\tilde{\mathcal{A}}f_m(X_s,Y_s)]ds \leq f_m(x,y).
\end{split}
\end{equation}
Let $T = \inf\{t>0: X_t=Y_t\}$. Since $(X,Y)$ is nonexplosive, $T_n\rightarrow T$ as $n\rightarrow \infty$. For any $x,y\in\Rnum_+^d$, it is easy to see that $f_m(x,y)\rightarrow|x-y|$ as $m\rightarrow\infty$. Letting $m\rightarrow\infty$ in \eqref{keyeq} and applying Fatou's lemma, we obtain that
\begin{equation*}
\Enum_{(x,y)}e^{\tilde{r}(t\wedge T_n)}|X_{t\wedge T_n}-Y_{t\wedge T_n}| \leq |x-y|.
\end{equation*}
Further letting $n\rightarrow\infty$ and applying Fatou's lemma give rise to
\begin{equation*}
\Enum_{(x,y)}e^{\tilde{r}(t\wedge T)}|X_{t\wedge T}-Y_{t\wedge T}| \leq |x-y|.
\end{equation*}
Let $Y'_t = Y_tI_{\set{t<T}}+X_tI_{\set{t\geq T}}$. Then
\begin{equation*}
\Enum_{(x,y)}e^{\tilde{r}t}|X_t-Y'_t|
= \Enum_{(x,y)}e^{\tilde{r}(t\wedge T)}|X_{t\wedge T}-Y_{t\wedge T}| \leq |x-y|.
\end{equation*}
Since $(X,Y')$ and $(X,Y)$ have the same marginal distributions, we finally obtain that
\begin{align*}
W(\delta_xP_t,\delta_yP_t) &\leq \Enum_{(x,y)}|X_t-Y'_t| \leq e^{-\tilde{r}t}|x-y|,
\end{align*}
which gives the desired result.
\end{proof}

\begin{lemma}\label{lemExpfinite}
Under the conditions in Theorem \ref{thmergodicity}, we have $\Enum_x|X_t|<\infty$ for any $x\in\Rnum_+^d$ and $t\geq 0$.
\end{lemma}

\begin{proof}
For any $n\geq 1$, we construct a function $g_n:\Rnum_+\rightarrow\Rnum_+$ satisfying
\begin{equation*}
g_n(t)=\begin{cases}
\frac{n}{2}t^2+\frac{1}{2n}, &0\leq t\leq\frac{1}{n},\\
t,&\frac{1}{n}<t\leq n,\\
n+\int_n^t(n+1-u)du,& n<t\leq a_n.
\end{cases}
\end{equation*}
where
\begin{equation*}
a_n = \frac{n+1+\sqrt{n^2+2n+5}}{2} > n+1.
\end{equation*}
It is easy to check that $g\in C^1[0,a_n]$ and $g'(a_n) = -1/a_n$. Moreover, the function $g_n$ can be constructed so that $g_n$ is decreasing over $[a_n,\infty)$, $g_n\in C_c^1(\Rnum_+)$, and
\begin{equation*}
|g_n(t)|\leq \frac{1}{t},\;\;\;\textrm{for any}\;t\geq a_n.
\end{equation*}
Let $f_n$ be a function on $\Rnum_+^d$ defined by $f_n(x) = g_n(|x|)$. Clearly, $f_n\in C_c^1(\Rnum_+^d)$ and $f_n(x)\rightarrow |x|$ for each $x\in\Rnum_+^d$. Next, we shall prove that there exists $C>0$ such that
\begin{equation}\label{boundedness}
\mathcal{A}f_n(x)\leq C,\;\;\;\textrm{for any}\;n\geq 1,x\in\Rnum_+^d.
\end{equation}
It is easy to see that $|\nabla f(x)| = |g_n'(|x|)|\leq 1$ for any $x\in\Rnum_+^d$. For any $|x|\leq1/n$, it follows from the mean value theorem that
\begin{equation*}
\mathcal{A}f_n(x) \leq |F(x)|+c(x)\int_{\Rnum_+^d}|y|\mu(dy)
\leq |F(0)|+L_F+[c(0)+L_c]\int_{\Rnum_+^d}|y|\mu(dy).
\end{equation*}
For any $1/n<|x|\leq n+1$, it follows from the dissipative condition that
\begin{equation*}
\begin{split}
\mathcal{A}f_n(x) &\leq \frac{g_n'(|x|)}{|x|}\langle F(x)-F(0),x\rangle+\frac{g_n'(|x|)}{|x|}\langle F(0),x\rangle+c(x)g_n'(|x|)\int_{\Rnum_+^d}|y|\mu(dy)\\
&\leq -rg_n'(|x|)|x|+|F(0)|+[c(0)+L_c|x|]g_n'(|x|)\int_{\Rnum_+^d}|y|\mu(dy)\\
&\leq |F(0)|+c(0)\int_{\Rnum_+^d}|y|\mu(dy).
\end{split}
\end{equation*}
In addition, it is easy to see that $g_n$ is decreasing over $[n+1,\infty)$ and $|g_n'(t)|\leq 1/t$ for any $t\geq n+1$. Thus, for any $|x|>n+1$,
\begin{equation}\label{equnibound4}
\mathcal{A}f_n(x) \leq |F(x)||\nabla f(x)| \leq \frac{|F(x)|}{|x|} \leq |F(0)|+L_F.
\end{equation}
The above three estimations imply \eqref{boundedness}. It thus follows from Fatou's lemma that
\begin{equation*}
\Enum_x|X_t| \leq \liminf_{n\rightarrow\infty}\Enum_xf_n(X_t)
= x+ \liminf_{n\rightarrow\infty}\int_0^t\Enum_x\mathcal{A}f_n(X_s)ds \leq x+Ct < \infty,
\end{equation*}
which gives the desired result.
\end{proof}

We are now in a position to prove Theorem \ref{thmergodicity}.

\begin{proof}[Proof of Theorem \ref{thmergodicity}]
Since we have proved Lemmas \ref{lemexpon} and \ref{lemExpfinite}, the rest of the proof follows the same line as \cite[Corollary 3]{eberle2016reflection}.
\end{proof}

\section{Proof of Theorem \ref{convergence}}\label{convergenceproof}
In this section, we shall prove the convergence of $X_V$ to $X$ as $V\rightarrow\infty$. For simplicity of notation, we only consider the case of $N = 1$, where the operator $\mathcal{A}$ has the form of \eqref{eqoperator}. The proof of the general case is totally the same.

To proceed, we recall the following two definitions \cite[Sections 3.7 and 1.5]{ethier2009markov}.

\begin{definition}
Let $S$ be a complete separable metric space and let $\{Y_V\}$ be a family of processes with sample paths in $D(\Rnum_+,S)$. If for every $\eta>0$ and $T>0$, there exists a compact set $\Gamma_{\eta,T}\subset S$ such that
\begin{equation*}
\inf_{V>0}P(Y_V(t)\in \Gamma_{\eta,T}\;\mbox{for any}\;0\leq t\leq T)\geq 1-\eta,
\end{equation*}
then we say that $\{Y_V\}$ satisfies the \emph{compact containment condition}.
\end{definition}

\begin{definition}
Let $\{T(t)\}$ be a measurable contraction semigroup on $B(S)$. Then the \emph{full generator} of $\{T(t)\}$ is defined as the set
\begin{equation*}
\mathcal{\widehat R} = \{(f,g)\in B(S)\times B(S): T(t)f-f=\int_0^tT(s)gds,\;\mbox{for any}\;t\geq 0\}.
\end{equation*}
\end{definition}

To prove weak convergence in the Skorohod space, we need the following lemma, which can be found in \cite[Corollaries 4.8.12 and 4.8.16]{ethier2009markov}.

\begin{lemma}\label{lemconverge}
Let $S$ be a complete separable metric space and let $\mathcal{\mathcal{R}}$ be an operator on $C_b(S)$. Suppose that for some $\nu\in\mathcal{P}(S)$, there exists a unique solution $Y$ to the martingale problem for $(\mathcal{R},\nu)$. For any $V>0$, let $Y_V$ be a c\`{a}dl\`{a}g Markov process with values in a set $S_V\subset S$ corresponding to a measurable contraction semigroup $\{T_V(t)\}$ with full generator $\mathcal{\widehat R}_V$. Suppose that $\{Y_V\}$ satisfies the compact containment condition and suppose that for each $f\in \mathcal{D}(\mathcal{A})$, there exists $(f_V,g_V)\in\mathcal{\widehat R}_V$ such that
\begin{equation*}
\sup_{V>0}\|f_V\|<\infty
\end{equation*}
and
\begin{equation}
\lim_{V\rightarrow\infty} \sup_{x\in S_V}|f_V(x)-f(x)|
= \lim_{V\rightarrow\infty}\sup_{x\in S_V}|g_V(x)-\mathcal{A}f(x)| = 0.
\end{equation}
Then $\nu_V\Rightarrow\nu$ as $V\rightarrow\infty$ implies $Y_V\Rightarrow Y$ in $D(\Rnum_+,S)$ as $V\rightarrow\infty$, where $\nu_V$ is the initial distribution of $Y_V$.
\end{lemma}

The following lemma plays a crucial role in studying the limit behavior of $X_V$.

\begin{lemma}\label{lemconverge2}
Suppose that the conditions in Theorem \ref{convergence} hold. Then for any $f\in\mathcal{D}(\mathcal{A})$,
\begin{equation}\label{eqAfconverge1}
\lim_{V\rightarrow\infty}\sup_{x\in E_V}|\mathcal{A}_Vf(x)-\mathcal{A}f(x)| = 0.
\end{equation}
\end{lemma}

\begin{proof}
Since $\beta_m$ is nonzero for a finite number of $m$, there exists $K>0$ such that $\beta_m\equiv 0$ for any $|m|\geq K$. Since $f$ has a compact support, there exists $\gamma>0$ such that $f(x)$ vanishes whenever $|x|\geq\gamma$. The above two facts suggest that $\mathcal{A}f(x) = 0$ for any $|x|\geq\gamma$ and $\mathcal{A}_Vf(n/V)=0$ for any $|n|\geq\gamma V+K$. Therefore, \eqref{eqAfconverge1} holds if and only if
\begin{equation}\label{equivalence}
\lim_{V\rightarrow\infty}\sup_{|n|\leq\gamma V+K}
\left|\mathcal{A}_Vf\left(\tfrac{n}{V}\right)-\mathcal{A}f\left(\tfrac{n}{V}\right)\right| = 0.
\end{equation}
It is easy to check that
\begin{equation*}
\left|\mathcal{A}_Vf\left(\tfrac{n}{V}\right)-\mathcal{A}f\left(\tfrac{n}{V}\right)\right|
\leq \textrm{I}+\textrm{II}+\textrm{III},
\end{equation*}
where
\begin{gather*}
\textrm{I} \leq \Big|\sum_{m\neq 0}
\left[\hat{q}_V\left(\tfrac{n}{V},\tfrac{n+m}{V}\right)-V\beta_m\left(\tfrac{n}{V}\right)\right]
\left[f\left(\tfrac{n+m}{V}\right)-f\left(\tfrac{n}{V}\right)\right]\Big|,\\
\textrm{II} = \Big|\sum_{m\neq 0}V\beta_m\left(\tfrac{n}{V}\right)
\left[f\left(\tfrac{n+m}{V}\right)-f\left(\tfrac{n}{V}\right)\right]
-\sum_{i=1}^dF_i\left(\tfrac{n}{V}\right)\partial_if\left(\tfrac{n}{V}\right)\Big|,\\
\textrm{III} = c\left(\tfrac{n}{V}\right)\Big|\sum_{m\in\Nnum^d}p(V,m)\left[f\left(\tfrac{n+m}{V}\right)
-f\left(\tfrac{n}{V}\right)\right]-\int_{\Rnum_+^d}\left[f\left(\tfrac{n}{V}+y\right)
-f\left(\tfrac{n}{V}\right)\right]\mu(dy)\Big|.
\end{gather*}
By the mean value theorem, we have
\begin{equation*}
\textrm{I} \leq \|\nabla f\|\sum_{m\neq 0}|m|
\left|\tfrac{1}{V}\hat{q}_V\left(\tfrac{n}{V},\tfrac{n+m}{V}\right)-\beta_m\left(\tfrac{n}{V}\right)\right|.
\end{equation*}
It thus follows from the condition \eqref{reactiongeneral} that
\begin{equation}\label{sum1}
\lim_{V\rightarrow\infty}\sup_{|n|\leq \gamma V+K}\textrm{I} = 0.
\end{equation}
By the mean value theorem, for any $0<|m|\leq K$, there exists $\theta_m\in(0,1)$ such that
\begin{equation*}
\textrm{II} \leq \sum_{m\neq 0}\sum_{i=1}^d|m_i|\beta_m\left(\tfrac{n}{V}\right)
\left|\partial_if\left(\tfrac{n+\theta_mm}{V}\right)-\partial_if\left(\tfrac{n}{V}\right)\right|
\end{equation*}
Since $\beta_m$ is locally bounded and $\partial_if$ is uniformly continuous, we have
\begin{equation}\label{sum2}
\lim_{V\rightarrow\infty}\sup_{|n|\leq \gamma V+K}\textrm{II} = 0.
\end{equation}
For any $\epsilon>0$, there exists $k>0$ such that $\mu\left((0,k]^d\right)>1-\epsilon$. For convenience, let $R_k = (0,[2kV])^d$ be a hypercube. When $V$ is sufficiently large, direct computations show that
\begin{equation*}
\begin{split}
\textrm{III} \leq&\; c\left(\tfrac{n}{V}\right)\bigg|\sum_{m\in R_k}p(V,m)
\left[f\left(\tfrac{n+m}{V}\right)-f\left(\tfrac{n}{V}\right)\right]
-\int_{R_k/V}\left[f\left(\tfrac{n}{V}+y\right)
-f\left(\tfrac{n}{V}\right)\right]\mu(dy)\bigg|\\
&\; +2\|f\|c\left(\tfrac{n}{V}\right)\sum_{m\notin R_k}p(V,m)
+2\|f\|c\left(\tfrac{n}{V}\right)\mu\left(\Rnum_+^d-(0,k]^d\right).
\end{split}
\end{equation*}
By the assumptions in \eqref{conditions}, we have
\begin{equation*}
\begin{split}
&\; \lim_{V\rightarrow\infty}\bigg|\sum_{m\notin R_k}p(V,m)-\mu(\Rnum_+^d-R_k/V)\bigg|\\
=&\; \lim_{V\rightarrow\infty}\bigg|\sum_{m\in R_k}p(V,m)-\mu(R_k/V)\bigg|
\leq \lim_{V\rightarrow\infty}\sum_{m\in R_k}\left|p(V,m)
-\mu\left[\tfrac{m}{V},\tfrac{m+1}{V}\right)\right|\\
\leq&\; (2k)^d\lim_{V\rightarrow\infty}V^d\sup_{0<|m|\leq 2kdV}\left|p(V,m)
-\mu\left[\tfrac{m}{V},\tfrac{m+1}{V}\right)\right| = 0.
\end{split}
\end{equation*}
When $V$ is sufficiently large, we have $\mu(\Rnum^d-R_k/V)\leq \mu(\Rnum^d-(0,k]^d)\leq\epsilon$ and thus
\begin{equation}\label{temp1}
\sum_{m\notin R_k}p(V,m) < 2\epsilon.
\end{equation}
Moreover, direct computations show that
\begin{equation*}
\begin{split}
&\; \bigg|\sum_{m\in R_k}p(V,m)\left[f\left(\tfrac{n+m}{V}\right)-f\left(\tfrac{n}{V}\right)\right]
-\int_{R_k/V}\left[f\left(\tfrac{n}{V}+y\right)-f\left(\tfrac{n}{V}\right)\right]\mu(dy)\bigg|\\
\leq&\; \bigg|\sum_{m\in R_k}\left[p(V,m)-\mu\left[\tfrac{m}{V},\tfrac{m+1}{V}\right)\right]
\left[f\left(\tfrac{n+m}{V}\right)-f\left(\tfrac{n}{V}\right)\right]\bigg|\\
&\; +\bigg|\sum_{m\in R_k}\mu\left[\tfrac{m}{V},\tfrac{m+1}{V}\right)
\left[f\left(\tfrac{n+m}{V}\right)-f\left(\tfrac{n}{V}\right)\right]
-\int_{R_k/V}\left[f\left(\tfrac{n}{V}+y\right)-f\left(\tfrac{n}{V}\right)\right]\mu(dy))\bigg|\\
\leq&\; 2(2k)^d\|f\|V^d\sup_{0<|m|\leq 2kdV}
\left|p(V,m)-\mu\left[\tfrac{m}{V},\tfrac{m+1}{V}\right)\right|\\
&\; +\int_{R_k/V}\left|f\left(\tfrac{n+[yV]}{V}\right)-f\left(\tfrac{n}{V}+y\right)\right|\mu(dy),
\end{split}
\end{equation*}
where $[yV] = ([y_1V],\ldots,[y_dV])$. When $V$ is sufficiently large, it follows from \eqref{conditions} and the uniform continuity of $f$ that
\begin{equation}\label{temp2}
\bigg|\sum_{m\in R_k}p(V,m)\left[f\left(\tfrac{n+m}{V}\right)
-f\left(\tfrac{n}{V}\right)\right] -\int_{R_k/V}\left[f\left(\tfrac{n}{V}+y\right)
-f\left(\tfrac{n}{V}\right)\right]\mu(dy)\bigg|\leq \epsilon.
\end{equation}
Combining \eqref{temp1} and \eqref{temp2} and noting that $c$ is continuous, we obtain that
\begin{equation}\label{sum3}
\lim_{V\rightarrow\infty}\sup_{|n|\leq\gamma V+K}\textrm{III} = 0.
\end{equation}
Finally, \eqref{equivalence} follows from \eqref{sum1}, \eqref{sum2}, and \eqref{sum3}.
\end{proof}

We are now in a position to prove Theorem \ref{convergence}.

\begin{proof}[Proof of Theorem \ref{convergence}]
For any $f\in\mathcal{D}(\mathcal{A}^\Delta)$, it is easy to check that $\mathcal{A}^\Delta f\in C((\Rnum_+^d)^\Delta)$ and thus $\mathcal{A}^\Delta$ is an linear operator on $C((\Rnum_+^d)^\Delta)$. Since $X$ is the unique solution to the martingale problem for $(\mathcal{A},\nu)$, it is easy to see that $X$ is also the unique solution to the martingale problem for $(\mathcal{A}^\Delta,\nu)$. Since $X_V$ is the unique solution to the martingale problem for $(\mathcal{A}_V,\nu_V)$, for any $f\in\mathcal{D}(\mathcal{A}_V)$, we have
\begin{equation*}
\Enum_xf(X_V(t)) = f(x)+\int_0^t\Enum_x\mathcal{A}_Vf(X_V(s))ds.
\end{equation*}
Thus $(f,\mathcal{A}_Vf)$ is in the full generator of $X_V$. If we take $S = (\Rnum_+^d)^\Delta$ and $S_V = E_V$, then $\{X_V\}$ automatically satisfies the compact containment condition since $S$ is compact. For any $f\in\mathcal{D}(\mathcal{A})$, setting $f_V = f|_{E_V}\in\mathcal{D}(\mathcal{A}_V)$ and applying Lemma \ref{lemconverge2}, we obtain that
\begin{equation}
\lim_{V\rightarrow\infty} \sup_{x\in E_V}|f_V(x)-f(x)|
= \lim_{V\rightarrow\infty}\sup_{x\in E_V}|\mathcal{A}_Vf_V(x)-\mathcal{A}f(x)| = 0.
\end{equation}
So far, all the conditions in Lemma \ref{lemconverge} have been checked and thus $X_V\Rightarrow X$ in $D(\Rnum_+,(\Rnum_+^d)^\Delta)$. Since both $X_V$ and $X$ have sample paths in $D(\Rnum_+,\Rnum_+^d)$, the desired result follows from \cite[Corollary 3.3.2]{ethier2009markov}.
\end{proof}

\section*{Acknowledgments}
The authors gratefully acknowledge Thomas G. Kurtz, David F. Anderson, and Hong Qian for helpful discussions. The authors are also greatly indebted to the anonymous referees for their valuable comments and suggestions which have greatly improved the presentation. X. Chen was funded by National Natural Science Foundation of China (Grant No. 11701483).

\setlength{\bibsep}{5pt}
\small\bibliographystyle{unrst}

\end{document}